\documentclass[leqno,10pt]{amsart} %
\usepackage{amssymb}
\usepackage{amstext}
\usepackage{graphicx}
\usepackage{tikz}

\usepackage[left=1.0in, right=1.0in, top=1.25in, bottom=1.25in]{geometry}
\usepackage{graphicx}
\usepackage{amsthm}
\usepackage{mathrsfs}
\usepackage{hyperref}

\newcommand{\gra}[1]{\raisebox{-.4cm}{\includegraphics[height=1cm]{UP#1.pdf}}}

\newcommand{\grb}[1]{\raisebox{-.8cm}{\includegraphics[height=2cm]{UP#1.pdf}}}

\newcommand{\gre}[1]{\raisebox{-2.3cm}{\includegraphics[height=5cm]{UP#1.pdf}}}

\theoremstyle{plain} 
\newtheorem{theo}{\indent\sc Theorem}[section]

\newtheorem{cor}[theo]{\indent\sc Corollary}
\newtheorem{prop}[theo]{\indent\sc Proposition}

\theoremstyle{definition} 
\newtheorem{defi}[theo]{\indent\sc Definition}

\newcommand{\C}{\mathbb C}

\newcommand{\ca}{\mathcal{C}}

\newcommand{\AC}{\mathcal{A}}
\newcommand{\Irr}{\text{Irr}(\mathcal{C})}

\begin{document}

\title[Quantum $G_{2}$ categories have property (T)]{Quantum $G_{2}$ categories have property (T)}
\author[Corey Jones]{\rm Corey Jones }

\keywords{ 
$C^{*}$ tensor category, Approximation/Rigidity Properties, Planar Algebras} 

\address[Corey Jones]{
Vanderbilt University\endgraf
Department of Mathematics\endgraf
Nashville\endgraf
USA
}
\email{corey.m.jones@vanderbilt.edu}
\thanks{The author was partially supported by NSF grant  DMS-1362138}
\maketitle

\begin{abstract} We show that the rigid C*-tensor categories of finite dimensional type 1 unitary representations of the quantum groups $U_{q}(\mathfrak{g}_{2})$ corresponding to the exceptional Lie group $G_2$ for positive $q\ne 1$ have property (T). 
\end{abstract}

\begin{section}{Introduction}

 Recently there has been a great deal of interest in approximation and rigidity properties for rigid $C^{*}$-tensor categories and subfactors.  Rigid $C^{*}$-tensor categories, introduced in their modern form by Longo and Roberts \cite{LR}, are structures which provide a unifying framework for  symmetries appearing in a variety of contexts.   They make a prominent appearance in the theory of compact quantum groups as representation categories \cite{NT}, and arise as DHR super-selection sectors in algebraic quantum field theory \cite{H}, hence are often described as encoding ``quantum symmetries'', generalizing the role of groups and their representations. 
 
 Rigid $C^{*}$-tensor categories are also realized as categories of bimodules appearing in the standard invariants of finite index subfactors \cite{Jo1}.  The standard invariant of the subfactor $N\subseteq M$ is the $2$-category of  bimodules that appear as tensor powers of $_{N} L^{2}(M)_{M}$ and its dual with respect to the relative tensor product. This yields two tensor categories, the $M$-$M$ bimodules and the $N$-$N$ bimodules, related by $_{N} L^{2}(M)_{M}$ and its dual.  The standard invariant is a powerful invariant of a subfactor, and was first axiomatized in full generality by Popa as standard $\lambda$-lattices \cite{Po2}.  Another useful realization is Jones' subfactors planar algebras \cite{Jo2}.  Alternatively, standard invariants can be axiomatized as rigid $C^{*}$-tensor categories along with a tensor generating $Q$-systems (see \cite{Mu1} for details), or via Ocneanu's paragroups in the finite depth case \cite{O0}.

 Popa introduced the concepts of approximation and rigidity properties for subfactors (see \cite{Po0}, \cite{Po1}, \cite{Po3}, \cite{Po4}). Popa's definitions can be formulated in terms of the symmetric enveloping inclusion  $T\subseteq S$ associated to $N\subseteq M$ (see \cite{Po0}, \cite{Po4}), and he showed that the definitions only depend on the standard invariant of the subfactor.  Recently, approximation and rigidity properties were translated from Popa's original definitions into the categorical setting by Popa and Vaes in \cite{PV}.  This allows for the definitions of property (T), the Haagerup property, and amenability to be defined in a conceptually uniform way for standard invariants and rigid $C^{*}$-tensor categories without reference to an ambient subfactor. They introduce a representation theory for standard invariants  and rigid $C^{*}$-tensor categories generalizing the unitary representation theory of groups, encoding in a natural way analytical properties.  To this end,  Popa and Vaes identify a class of \textit{admissible representations} of the fusion algebra of a category.  This class admits a universal representation, allowing for the construction of a univsersal algebra $C^{*}(\ca)$, generalizing the universal $C^{*}$-algebra for groups.

Soon after the initial work of Popa and Vaes, admissible representations of the fusion algebra were interpreted from different points of view.  Neshveyev and Yamashita showed that admissible representations can be understood as objects in the Drinfeld center of the ind-category \cite{NT}.  In \cite{GJ}, the authors show that admissible representations have a natural interpretation in the annular representation theory for planar algebras of Jones \cite{Jo3}, \cite{Jo4} and the representation theory of the tube algebra of a category, introduced by Ocneanu in \cite{O}.  The fusion algebra is a corner of the tube algebra, and in \cite{GJ} it is shown that admissible representations of the fusion algebra are precisely representations which are restrictions of representations of the whole tube algebra.  Since the tube algebra is computable in principle, this provides a method for determining admissible representations. 

While the notion of property (T) for subfactor standard invariants (due to Popa) has been in existence for many years, examples have been somewhat elusive.  Until recently, the only known examples of subfactor standard invariants with property (T) came in some way from discrete (T) groups.  In particular, the diagonal subfactors and the Bisch-Haagerup subfactors, when constructed with property (T) groups, produce property (T) standard invariants (see Popa \cite{Po3}, \cite{Po4} and Bisch-Popa, \cite{BiPo} respectively).  Arano showed in \cite{Ar} that the discrete dual of the compact quantum groups $SU_{q}(N)$ have central property (T) for $N\ge 3$ odd and positive $q\ne1$, which Popa and Vaes showed is equivalent to the corresponding representation category of $SU_{q}(N)$ having property (T) (see \cite{PV}).  From these categories, one can construct subfactors whose even bimodule categories are equivalent to $Rep(SU_{q}(N))$, providing the first examples of subfactors not coming in some way from discrete groups whose standard invariants have (T).  In light of this breakthrough, it is natural to wonder if other quantum group categories may have property (T).

The categories $\mathscr{C}_{q}(\mathfrak{g}_{2})$ for positive $q$ are the rigid $C^{*}$-tensor categories of finite dimensional type 1 unitary representations of the Drinfeld-Jimbo quantum groups, corresponding to the exceptional Lie group $G_2$ (see \cite{NT}).  They have been described diagrammatically by Kuperberg \cite{K}, \cite{K2}.  These diagrammatic categories, denoted $(G_{2})_{q}$, have been further studied by Morrsion, Peters, and Snyder in \cite{MPS}, where they appear as ``small'' examples of trivalent categories in their classification program, and a description of small idempotents is obtained.

In this note, we show that the categories $\mathscr{C}_{q}(\mathfrak{g}_{2})$ have property (T) for positive $q\ne 1$ using their diagrammtic descriptions.  This provides a  new class of examples of subfactors with property (T) standard invariant. We give a brief outline of the argument:

  Since the fusion algebra of $(G_{2})_{q}$ is abelian (the category is braided), the universal $C^{*}$-algebra is isomorphic to the algebra continuous functions on its spectrum, which corresponds to the space of irreducible admissible representations.  Having property $(T)$ in this setting simply translates to the trivial representation being isolated in the spectrum.  Identifying the fusion algebra $(G_{2})_{q}$ as polynomials in two self adjoint variables, the possible irreducible representations correspond to points in the plane, defined by evaluation of polynomials.  Using some general restrictions, we reduce the possibilities for admissible representations to a rectangle, with the trivial representation at a corner.  Applying the description of small minimal idempotents provided by Morrison, Peters, and Snyder \cite{MPS}, we define a function of the plane $f(\alpha, t)$ which is $0$ at the trivial representation, and has the property that this function must be non-negative at $(\alpha, t)$ for the representation corresponding to that point to be admissible.  Then using elementary calculus, we show in a neighborhood of the trivial representation this function is strictly negative, implying that $(G_{2})_{q}$ has property (T).  We note our calculus arguments break down precisely when $q=1$, which is to be expected since the classical $G_{2}$ representation category is amenable.

 Although the categories we study are quantum group category, we emphasis our proof mostly uses only the basic skein theoretic description of the category.  We also remark that a surprisingly small amount of data from the actual category is used.  We hope this note demonstrates the computational usefulness of diagrammatics for studying analytical properties of categories, which was also demonstrated in \cite{Bro}, where the authors give a direct planar algebraic proof that the categories $TLJ(\delta)$ have the Haagerup property for all  $\delta\ge 2$ using Popa's original definitions. 

The outline of the paper is as follows: We begin by describing rigid $C^{*}$-tensor categories and the tube algebra.  We then discuss admissible representations and property (T).  Finally we give a short proof that $(G_{2})_{q}$ has property (T).  In the appendix, we discuss in greater detail the quantum group $U_{q}(\mathfrak{g}_{2})$ and describe how the $*$-structure from $\mathscr{C}_{q}(\mathfrak{g}_2)$ transports to the planar algebra $(G_{2})_{q}$.  

\medskip

\textbf{Acknowledgements:}  The author would like to thank Noah Snyder and Yunxiang Ren for useful discussions on $(G_{2})_{q}$, and Vaughan Jones for his encouragement. He also thanks Marcel Bischoff, Arnaud Brothier, and Michael Northington for helpful comments, and Shamindra Ghosh for many interesting discussions.  The author was supported by NSF grant  DMS-1362138.

\end{section}

\begin{section}{Preliminaries}  
\begin{subsection}{Rigid C*-Tensor Categories}

In this paper we will be concerned with semi-simple, $C^{*}$-categories with strict tensor functor, simple unit and duals.  We also assume that $\ca$ has countably many isomorphism classes of simple objects.  Often in the literature, this is the definition of a rigid $C^{*}$-tensor category.  We briefly elaborate on the meaning of each of these words.

\ \ A $C^{*}$-category is a $\C$-linear category $\ca$, with each morphism space $Mor(X,Y)$ a Banach space, and a conjugate linear, involutive, contravariant functor $*:\mathcal{C}\rightarrow \ca$ which fixes objects and satisfies for every morphism $f$, $||f^{*}f||=||ff^{*}||=||f||^{2}$.  We say the category is \textit{semi-simple} if the category has direct sums, sub-objects, and each $Mor(X,Y)$ is finite dimensional.

A strict tensor functor is a bi-linear functor $\otimes: \ca \times \ca\rightarrow \ca$, which is associative and has a distinguished unit $id\in Obj(\ca)$ such that $X\otimes id=X=id \otimes X$. The category is \textit{rigid} if for each $X\in Obj(\ca)$, there exists $\overline{X}\in Obj(\ca)$ and morphism $R\in Mor(id, \overline{X}\otimes X)$ and $\overline{R}\in Mor(id, X\otimes \overline{X})$ satisfying the so-called conjugate equations:

$$(1_{\overline{X}}\otimes \overline{R}^{*})(R\otimes 1_{\overline{X}})=1_{\overline{X}}\ \text{and}\ (1_{X}\otimes R^{*})(\overline{R}\otimes 1_{X})=1_{X}$$

\ \ We say two objects are $X,Y$ are \textit{isomorphic} if there exists $f\in Mor(X,Y)$ such that $f^{*}f=1_{X}$ and $f f^{*}=1_{Y}$.  We call an object $X$ \textit{simple} if $Mor(X,X)\cong \mathbb{C}$.  We note that for any simple objects $X$ and $Y$, $Mor(X,Y)$ is either isomorphic to $\mathbb{C}$ or $0$. Two simple objects are isomorphic if and only if $Mor(X,Y)\cong \mathbb{C}$. Isomorphism defines an equivalence relation on the collection of all objects and we denote the equivalence class of an object by $[X]$, and the set of isomorphism classes of simple objects $\text{Irr}(\mathcal{C})$.

\ \ The semi-simplicity axiom implies that for any object $X$, $Mor(X,X)$ is a finite dimensional  $C^{*}$-algebra over $\mathbb{C}$, hence a multi-matrix algebra.  It is easy to see that each summand of the matrix algebra corresponds to an equivalence class of simple objects, and the dimension of the matrix algebra corresponding to a simple object $Y$ is the square of the multiplicity with which $Y$ occurs in $X$.  In general for a simple object $Y$ and any object $X$, we denote by $N^{Y}_{X}$ the natural number describing the multiplicity with which $[Y]$ appears in the simple object decomposition of $X$.  If $X$ is equivalent to a subobject of $Y$, we write $X\prec Y$.  We often write $X\otimes Y$ simply as $XY$ for objects $X$ and $Y$.

 \ \ For two simple objects $X$ and $Y$, we have that $[X\otimes Y]\cong \oplus_{Z} N^{Z}_{XY}[Z]$.  These means that the the tensor product of $X$ and $Y$ decomposes as a direct sum of simple objects of which $N^{Z}_{XY}$ are equivalent to the simple object $Z$.  The $N^{Z}_{XY}$ specify the \textit{fusion rules} of the tensor category and are a critical piece of data.

The fusion algebra is the complex linear span of isomorphism classes of simple objects $\C[\text{Irr}(\ca)]$, with multiplication given by linear extension of the fusion rules.  This algebra has a $*$-involution defined by $[X]^{*}=[\overline{X}]$ and extended conjugate-linearly.  This algebra is a central object of study in approximation and rigidity theory for rigid $C^{*}$-tensor categories.  

\ \ For a more detailed discussion and analysis of the axioms of a rigid $C^{*}$-tensor category, see the paper of Longo and Roberts \cite{LR}, where these categories were first defined and studied with this axiomatization.

\ \  In a rigid $C^{*}$-tensor category, we can define the statistical dimension of an object $d(X)=inf _{(R,\overline{R})} || R|| ||\overline{R}||$, where the infimum is taken over all solutions to the conjugate equations for an object $X$.  The function $d(\ .\ ):Obj(\ca)\rightarrow \mathbb{R}_{+}$ depends on objects only up to unitary isomorphism.  It is multiplicative and additive and satisfies $d(X)=d(\overline{X})$ for any dual of $X$.  We called solutions to the conjugate equations \textit{standard} if $||R||=||\overline{R}||=d(X)^{\frac{1}{2}}$, and such solutions are essentially unique.  For standard solutions of the conjugate equations, we have a well defined trace $Tr_{X}$ on endomorphism spaces $Mor(X,X)$ given by

 $$Tr_{X}(f)=R^{*}(1_{\overline{X}}\otimes f)R=\overline{R}^{*}(f\otimes 1_{\overline{X}})\overline{R}\in Mor(id,id)\cong \C$$  

This trace does not depend on the choice of dual for $X$ or on the choice of standard solutions.  We note that $Tr(1_{X})=d(X)$.  See \cite{LR} for details.

\end{subsection}

\begin{subsection}{The Tube Algebra}
The tube algebra $\AC$ of a semi-simple rigid tensor category $\ca$ was introduced by Ocneanu in \cite{O}.  This algebra has proved useful for computing the Drinfeld center $Z(\ca)$, since finite dimensional irreducible representations of $\AC$ are in one-to-one correspondence with simple objects of $Z(\ca)$ (see \cite{I} and \cite{Mu2}).  Stefaan Vaes has pointed out that in general, arbitrary representations of $\AC$ are in one-to-one correspondence with objects in the category $Z(\text{ind-}\ca)$ studied by Neshveyev and Yamashita in \cite{NT}.

   For a rigid $C^{*}$-tensor category $\ca$, choose a representative $X_{k}\in k$ for $k\in \Irr$.  We choose the strict tensor identity $id$ to represent its class and label it with index $0$, so that $X_{0}=id$.

The tube algebra is defined as the algebraic direct sum (finiteley many non-zero terms)

$$\AC:=\oplus_{i,j,k\in \Irr} Mor(X_{k}\otimes X_{i}, X_{j}\otimes X_{k})$$

An element $x\in \AC$  is given by a sequence $x^{k}_{i,j}\in Mor(X_{k}\otimes X_{i}, X_{j}\otimes X_{k})$ with only finitely many terms non-zero.  Recall for a simple object $\alpha$ and arbitrary $\beta\in Obj(\ca)$, $Mor(\alpha, \beta)$ has a Hilbert space structure with inner product defined by $\eta^{*}\xi=\langle \xi, \eta\rangle  1_{\alpha}$.  Note that this inner product differs by the tracial inner product by a factor of $d(\alpha)$.

$\AC$ carries the structure of an associative $*$-algebra, defined by

$$(x\cdot y)^{k}_{i,j}=  \sum_{s, m, l\in \Irr}\ \sum_{V\in onb(X_k,\ X_m\otimes X_l)}(1_{j}\otimes V^{*}) ( x^{m}_{s,j}\otimes 1_{l})(1_{m}\otimes y^{l}_{i,s})(V\otimes 1_{i})$$

$$(x^{\#})^{k}_{i,j}=  (\overline{R}^{*}_{k}\otimes 1_{j}\otimes 1_{k}) (1_{k}\otimes (x^{\overline{k}}_{j,i})^{*}\otimes 1_{k})(1_{k}\otimes 1_{i}\otimes R_{k})$$

\medskip

\noindent where $R_{k}\in Mor(id, \overline{X}_{k}\otimes X_{k})$ and $\overline{R}_{k}\in   Mor(id, X_k\otimes \overline{X}_{k})$ are standard solutions to the conjugate equations for $X_{k}$.  We remark that we denote the $*$-involution by $\#$ to avoid confusion with the $*$-involution from $\ca$.  In the first sum above, $onb$ denotes an orthonormal basis with respect to the inner product described above, and we may have $onb(X_k,\ X_ m\otimes X_ l)=\varnothing$ if $X_{k}$ is not  a sub-object of $X_{m}\otimes X_{l}$.  We mention the above compact form for the definition of the tube algebra was borrowed from Stefaan Vaes.

We define the subspaces $\AC^{k}_{i,j}:=Mor(X_{k}\otimes X_{i}, X_{j}\otimes X_{k})\subset \AC$.  For arbitrary objects $\alpha\in Obj(\ca)$, we have a natural map $\Psi: Mor(\alpha\otimes X_{i}, X_{j}\otimes\alpha)\rightarrow \AC$ given by $$\displaystyle \Psi(f)=\sum_{k \prec \alpha}\sum_{V\in onb(k, \alpha)}   (1_{j}\otimes V^{*})f(V\otimes 1_{i}).$$

We will use this map later, in our analysis of the categories $(G_{2})_{q}$.  For each $k\in \Irr$, there is a projection $p_{m}\in \AC^{0}_{m,m}$ given by $p_{m}:=1_{m}\in Mor(id\otimes X_{m}, X_{m}\otimes id)\in \AC$.  In particular $(p_{m})^{k}_{i,j}=\delta_{k,0}\delta_{i,j}\delta_{j,m} 1_{m}$.

We see that $\AC_{0,0}=p_{0}\AC p_{0}$ is a corner of the tube algebra.  This corner is a unital $*$-algebra.  Recall the fusion algebra of $\ca$ is the complex linear span of isomorphism classes of simple objects $\C[\text{Irr}(\ca)]$.  Multiplication is the linear extension of fusion rules and $*$ is given on basis elements by the duality.  From the definition of multiplication in $\AC$, one easily sees the following:

\begin{prop}The fusion algebra $\C[\text{Irr}(\ca)]$ is $*$-isomorphic to $\AC_{0,0}$, via the map $[X_{k}]\rightarrow 1_{k}\in (X_{k}\otimes id, id\otimes X_{k})\in \AC^{k}_{0,0}$.
\end{prop}  

\end{subsection}

\begin{subsection}{Representations and property (T)}

\begin{defi} $Rep(\AC)$ is the category whose objects are non-degenerate $*$-homomorphisms $\pi:\AC\rightarrow B(H)$ for some Hilbert space $H$, and whose morphisms are bounded intertwiners.
\end{defi}

This category can be given the structure of a braided $C^{*}$-tensor category, though in general it is not rigid.  In the case that $\ca$ is fusion ($|\Irr|<\infty$), the category of finite dimensional representations is tensor equivalent to the Drinfeld center $Z(\ca)$ (see \cite{DGG} for details).

\begin{defi} A $*$-homomorphism $\pi:\C[\text{Irr}(\ca)]\rightarrow B(H)$ of the fusion algebra is \textit{admissible} if there exists a $\hat{\pi}\in Rep(\AC)$ such that $\hat{\pi}|_{\AC_{0,0}}$ is unitarily equivalent to $\pi$.
\end{defi}

We remark that this is not the original definition of admissible representation given by Popa and Vaes in \cite{PV}, but is equivalent by  \cite{GJ}, Corollary 6.9.

\begin{prop} (\cite{GJ}, Corolloary 4.8) Let $\phi: \AC_{0,0}\rightarrow \C$ be a linear functional.  The following are equivalent:
\begin{enumerate}
\item
$\phi$ is a vector state in an admissible representation.
\item
$\phi(p_{0})=1$, and $\phi(x^{\#}\cdot x)\ge 0$ for all $x\in \AC_{0,k}$ and $k\in \Lambda$.
\end{enumerate}
\end{prop}

We call the collection of functionals satisfying the equivalent conditions of the above proposition \textit{annular states}, denoted $\Phi_{0}(\AC)$.  The word annular comes from the correspondence between representations of the tube algebra and representations of the affine annular category of a planar algebra introduced by Jones (see \cite{Jo3}, \cite{Jo4}, \cite{DGG}, \cite{GJ}).  The idea is that annular states play a similar role in our representation theory to positive definite functions in the representation theory of groups.

Every $\phi\in \ell^{\infty}(\text{Irr}(\ca))$ defines a linear functional $\hat{\phi}$ on the fusion algebra $\C[\text{Irr}(\ca)]$ given by $$\hat{\phi}(\sum_{k} c_{k}[X_{k}])=\sum_{k} c_{k}d(X_{k})\phi([X_{k}])$$

\begin{defi} $\phi\in \ell^{\infty}(\text{Irr}(\ca))$ is a \textit{cp-multiplier} if $\hat{\phi}$ is an annular state.
\end{defi}

Again, this is not the original definition of cp-multiplier introduced by Popa and Vaes.  For their original definition see \cite{PV}, Defintion 3.1. That the definition presented here is equivalent to theirs follows from \cite{GJ}, Theorem 6.6.   Without going into the details of the original definition, we describe its motivation as follows:

\ \ Associated to a subfactor $N\subseteq M$, Popa defined a von Neumann algebra inclusion $T\subseteq S$, called the symmetric enveloping inclusion \cite{Po0}.  Popa gave definitions for analytical properties of a subfactor in terms of sequences of UCP maps $\phi_{i}:S\rightarrow S$ which are $T$ bimodular, satisfying certain convergence properties (see \cite{Po3}).  It turns out that decomposing $L^{2}(S)$ as a $T-T$ bimodule shows that such maps are  determined by functions on $\text{Irr}(\ca)$, where $\ca$ is the category of ``even" bimodules of the subfactor.  Their definition of $cp$-multiplier is a necessary and sufficient condition for a function  $\phi\in \ell^{\infty}(\text{Irr}(\ca))$ to produce a $T$-$T$ bimodular UCP map on $S$ in the prescribed way.  Furthermore, conditions for approximation and rigidity properties on the sequence of UCP maps can be directly translated into conditions on the corresponding functions in $\ell^{\infty}(\text{Irr}(\ca))$ with out reference to the symmetric enveloping inclusion.  This results in definitions for approximation and rigidity properties for the category $\ca$ without reference to a subfactor.

\begin{defi} (\cite{PV}, Definition 5.10)  A rigid $C^{*}$-tensor category has \textit{property (T)} if every sequence of cp-multipliers $\phi_{i}$ which converges to $1$ pointwise on $\text{Irr}(\ca)$ converges uniformly.
\end{defi}

  There exists a $C^{*}$-closure, $C^{*}(\ca)$, of the fusion algebra introduced in \cite{PV} such that continuous Hilbert space representations of this $C^{*}$-algebra are in one-to-one correspondence with admissible representations.    It can be shown for $X\in \text{Irr}(\ca)$, in any admissible representation $\pi$ we have $\| \pi(X)\| \le d(X)$ (see for example \cite{PV}, Proposition 4.2 or \cite{GJ}, Lemma 4.4).

This bound permits the definition of a universal admissible representation $\pi_{u}$.  $C^{*}(\ca)$ is defined by taking the $C^{*}$-completion of $\pi_{u}(\C[\text{Irr}(\ca)])$.  The norm of an element in the fusion algebra can be written

$$\|f\|_{u}=\sup_{\phi\in \Phi_{0}(\AC)}\phi(x^{\#}\cdot x)^{\frac{1}{2}}.$$

  We remark that in fact we can define such a universal representation and $C^{*}$-algebra for the entire tube algebra $\AC$ but we do not need this here (see \cite{GJ}, Definition 4.6).

\  \ As with groups, there are many equivalent characterizations of property (T). We record the following provided by Popa and Vaes:

\begin{prop}(\cite{PV}, Proposition 5.5)  $\ca$ has property (T) if and only if there exists a projection $p\in C^{*}(\ca)$ such that $\alpha p=d(\alpha)p$ for all $\alpha\in \text{Irr}(\ca)$. 
\end{prop}

The so-called trivial representation $1_{\ca}$ is the one dimensional representation of the fusion algebra spanned by $v_{0}$ such that $1_{\ca}([X])v_{0}=d(X)v_{0}$ for all $X\in \text{Irr}(\ca)$.  Viewing $1_{\ca}$ as an annular state (\cite{GJ}, Lemma 4.11), the corresponding cp-multiplier is the constant function $1$ in $\ell^{\infty}(\text{Irr}(\ca))$.

\ \ For categories with abelian fusion rules (for example, all braided categories), $C^{*}(\ca)\cong C(Z)$ for some compact Hausdorff space $Z$.  Points in $Z$ correspond to one-dimensional representations of the fusion algebra, so $1_{\ca}\in Z$.  We have the following easy consequence of the above proposition:

\begin{cor} If $\ca$ has abelian fusion rules so that $C^{*}(\ca)\cong C(Z)$ for some compact Hausdorff space $Z$, then $\ca$ has property (T) if and only if the trivial representation $1_{\ca}$ is isolated in $Z$
\end{cor}

\begin{proof} If $\ca$ has abelian fusion rules $C^{*}(\ca)\cong C(Z)$, where $Z$ is the spectrum of $C^{*}(\ca)$.   If $1_{\ca}$ is isolated in the spectrum, then the characteristic function $\delta_{\{1_{\ca}\}}\in C(Z)\cong C^{*}(\ca)$ is a projection satisfying the required property.  Conversely, if we had such a projection $p$, then it could be represented by the characteristic function of some clopen set $Y\subseteq Z$.  Since $\alpha p=d(\alpha) p$, this implies that when viewing an object $\alpha$ as a function on $Z$,  $\alpha|_{Y}=d(\alpha)=\alpha(1_{\ca})$.  Extending by linearity, we see that for an arbitrary element in the fusion algebra  $\beta$, $\beta|_{Y}=1_{\ca}(\beta)$.  This equality extends to the $C^{*}$-closure $C^{*}(\ca)\cong C(Z)$.  Since the points of $Y$ are not separated by $C(Z)$ from $1_{\ca}$, by the Stone-Weierstrass theorem we have $Y=\{1_{\ca}\}$, hence $\{1_{\ca}\}$ is clopen, hence $1_{\ca}$ is isolated in $Z$.
\end{proof}

\end{subsection}
\end{section}

\begin{section}{$(G_{2})_{q}$ categories}

There are many ways to describe rigid $C^{*}$-tensor categories.  One of the most useful is the planar algebra approach introduced by Jones \cite{Jo2}.  The idea is to use formal linear combinations of planar diagrams to represent morphisms in your category.  These diagrams satisfy some linear dependences called skein relations in modern parlance.  The most famous skein relations are the ones defining the Jones and HOMFLY polynomials.

The $(G_{2})_{q}$ categories we describe are a particularly nice type of planar algebra called a trivalent category. These were introduced in their current form by Morrison, Peters, and Snyder \cite{MPS}.  Using dimension restrictions on morphism spaces as a notion of ``small",  they were able to classify the ``smallest'' examples.  The $(G_{2})_{q}$ categories appear in their classification list.

\begin{defi} ( \cite{MPS}, Definition 2.4)  A \textit{trivalent category} $\ca$ is a non-degenerate, evaluable, pivotal category over $\C$, with a tensor generating object $X$ satisfying dim $Mor(id, X)=0$, dim $Mor(id, X\otimes X)=1$, dim $Mor(id, X\otimes X\otimes X)=1$, generated (as a planar algebra) by a trivalent vertex for $X$ .
\end{defi}

We summarize the basic properties of trivalent categories:

\begin{enumerate}
\item
Objects in the category can be represented by $\mathbb{N}\cup \{0\}$, and correspond to tensor powers of a generating object $X$.
\item
$Mor(k,m)$ is the complex linear span of isotopy classes of planar trivalent graphs embedded in a rectangle, with $m$ boundary points on the top of the rectangle, $k$ boundary points on the bottom, and no boundary points on the sides of the rectangle.  These diagrams are subject to \textit{skein relations}, which are linear dependences among the trivalent graphs which make $Mor(k,m)$ finite dimensional. (Note: We consider graphs with no vertices at all, namely line segments attached to the boundaries, as trivalent graphs)
\item
$Mor(0,0)\cong \mathbb{C}$.  In other words, our skein relations reduce every closed trivalent graph to a scalar multiple of the empty trivalent graph.  Identifying the empty graph with $1\in \C$, this means we have associated to every closed trivalent graph a complex number.
\item
Composition of morphisms is vertical stacking of rectangles.
\item
Tensor product on objects is addition of natural numbers, on morphisms it is horizontal stacking of rectangles.
\item
Duality is given by rotation by $\pi$ or $-\pi$ (these manifestly agree in our setting). 
\item
The linear functional $Tr:Mor(k,k)\rightarrow \C$ given by connecting the top strings of the rectangle to the bottom is non-degenerate.
\end{enumerate}

\begin{defi} A trivalent category is a \textit{$C^{*}$-trivalent category} if the maps $*: Mor(k,m)\rightarrow Mor(m,k) $ given by reflecting graphs across a horizontal line and conjugating complex coefficients are well-defined modulo the skein relations, and $Tr(x^{*}x)\ge 0$ for every $x\in Mor(k,m)$, and $k,m\in \mathbb{N}\bigcup\{0\}$.
\end{defi}

From a $C^{*}$-trivalent category, we can construct a rigid $C^{*}$-tensor category as follows:
First, it can be shown that a category satisfying all these conditions has a negligible category ideal, generated by diagrams with $Tr(x^{*}x)$=0.  Quotienting by this produces a trivalent category with condition $8$ replaced by $Tr(x^{*}x)>0$ .
Next, we take the projection completion.  Objects in this category will be projections living in some $Mor(k,k)$. For two projections $P\in Mor(k,k),\ Q\in Mor(m,m)$, $Mor(P,Q)=\{f\in Mor(k,m)\ :\ QfP=f\}$.  Now we formally add direct sums to the category.  The resulting category will have objects direct sums of projections, and morphisms matrices of the morphisms between projections.  The result is a rigid $C^{*}$-tensor category, which we also call $\ca$.  

Notice the duality map we have defined is automatically pivotal.  Also the strict tensor identity $id$ is given by the empty diagram.  Another consequence of the definitions is that the generating object $X$ is symmetrically self-dual (see \cite{BHP}, Definition 2.10).

The $(G_{2})_{q}$ trivalent categories were introduced by Kuperberg in \cite{K} and \cite{K2}.  Kuperberg showed that these categories are equivalent to the category of (type 1) finite dimensional representations of the Drinfeld-Jimbo quantum groups $U_{q}(\mathfrak{g}_{2})$.

To define a trivalent category, it suffices to specify a set of skein relations.  In general it is a difficult problem to determine whether an set of skein relations produces a trivalent category.  In particular, one has to verify that your relations are consistent and evaluable.  Otherwise you may end up with $Mor(0,0)$ being $0$ dimensional, or with infinite dimensional morphism spaces.  Kuperberg showed the following skein relations are indeed consistent and evaluable, resulting in a trivalent category.  The skein theory we present for $(G_{2})_{q}$ can be found in \cite{MPS}, Definition 5.21.  It differs from Kuperberg's description in two ways:  The trivalent vertex is normalized, and the $q^{2}$ here is Kuperberg's $q$.

\begin{defi} $(G_{2})_{q}$ for strictly positive $q$ is the trivalent category defined by the following skein relations:
\end{defi}

$$\gra{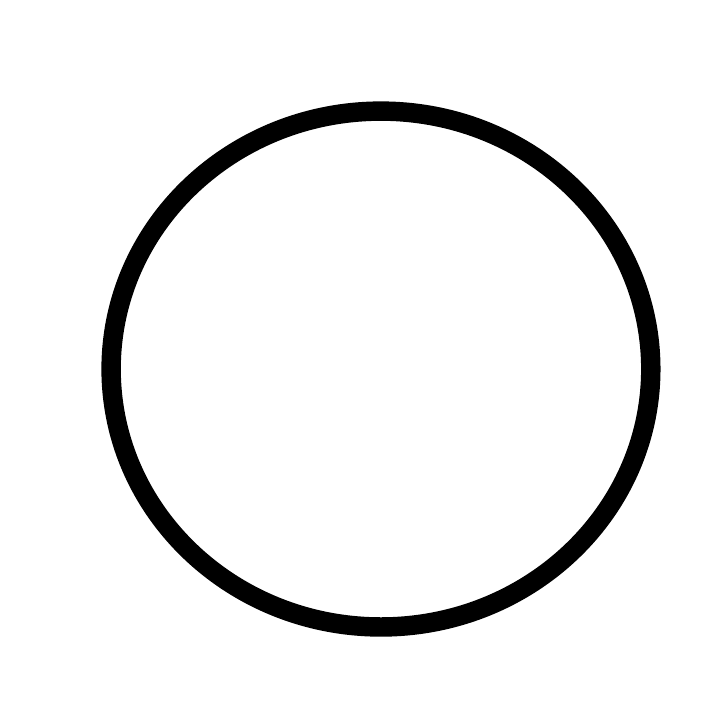}=\delta:=q^{10}+q^{8}+q^{2}+1+q^{-2}+q^{-8}+q^{-10}$$
$$\gra{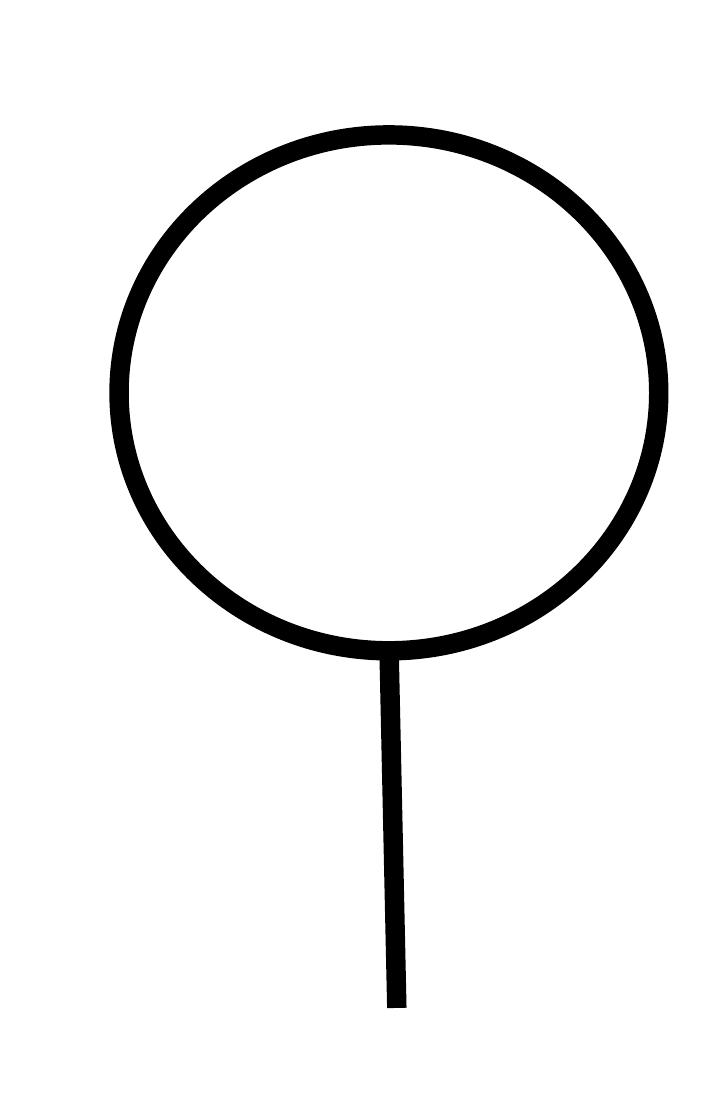}=0$$
$$\grb{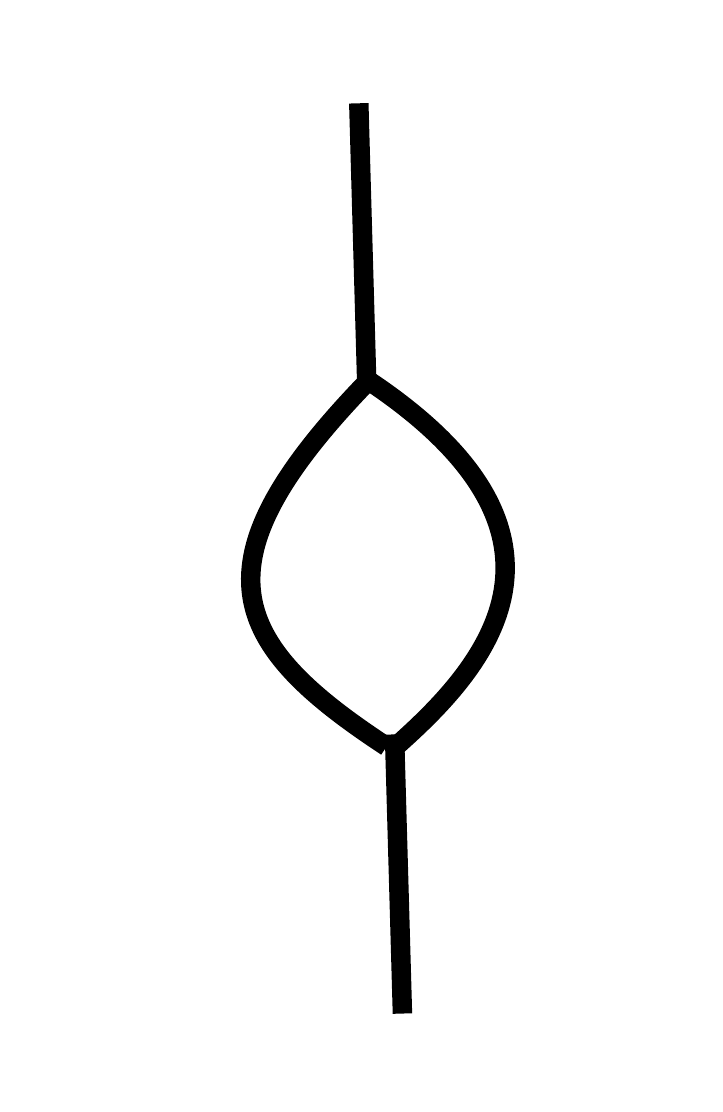}=\grb{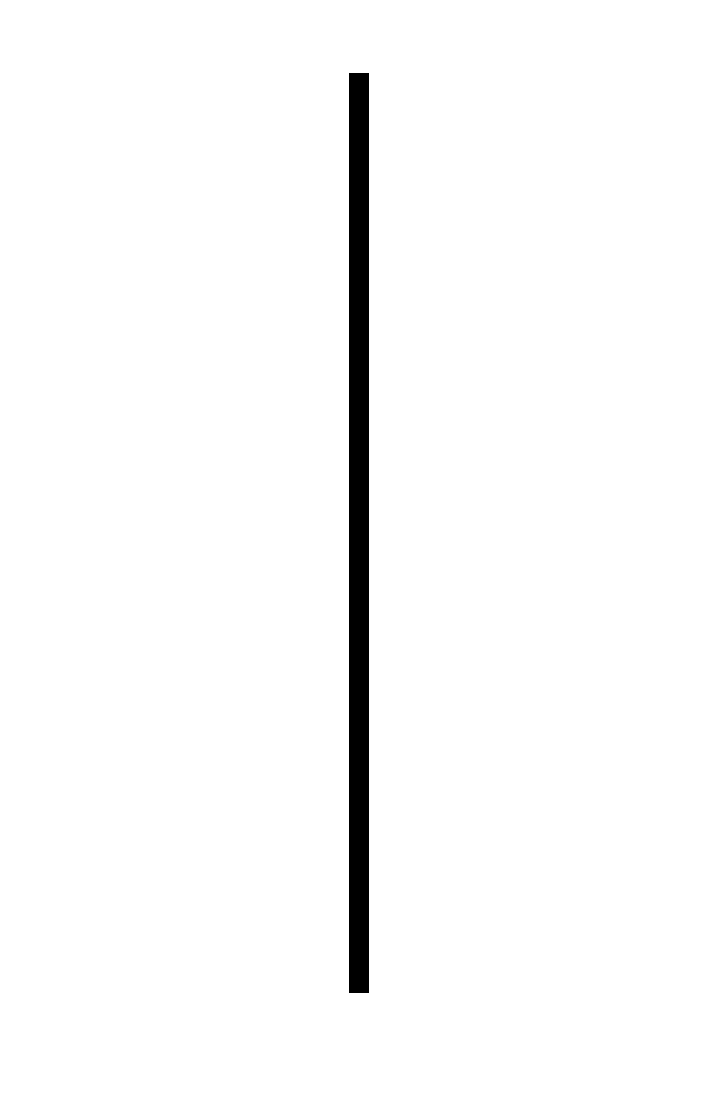}$$
$$\grb{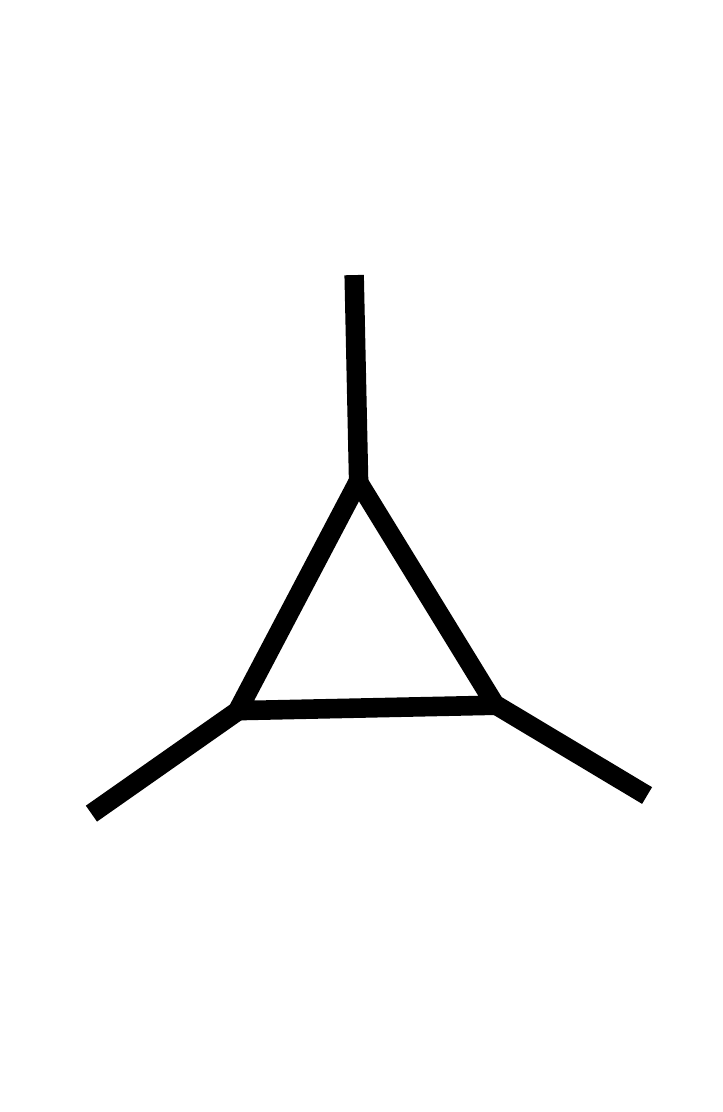}=c\grb{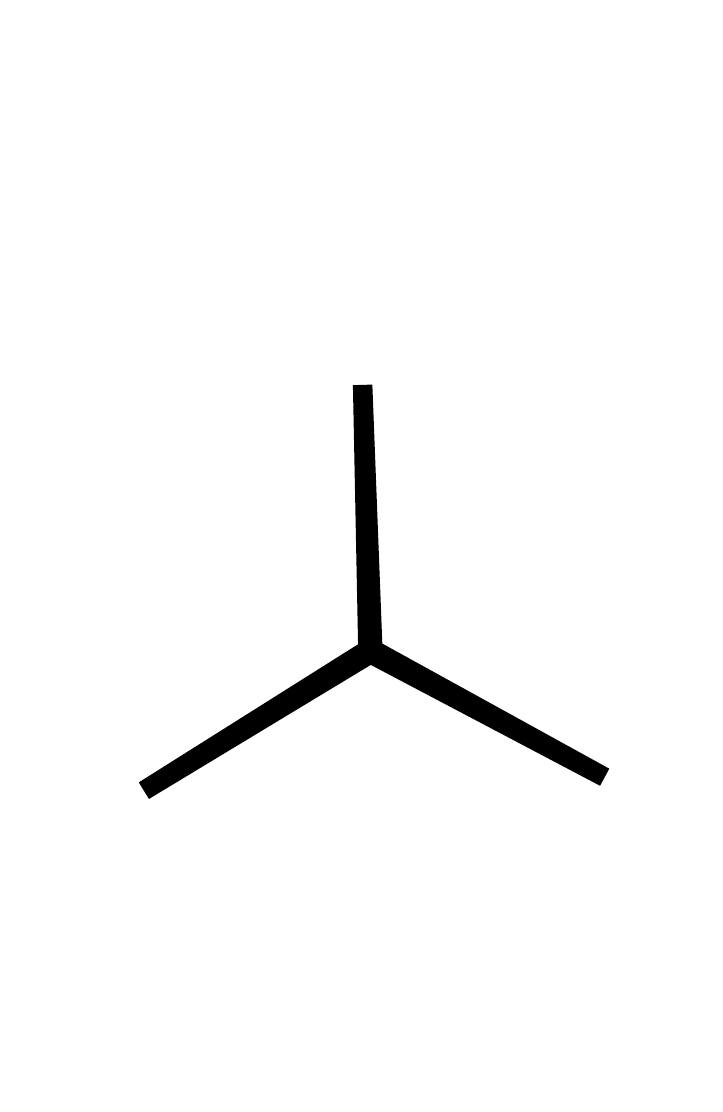}$$
$$\grb{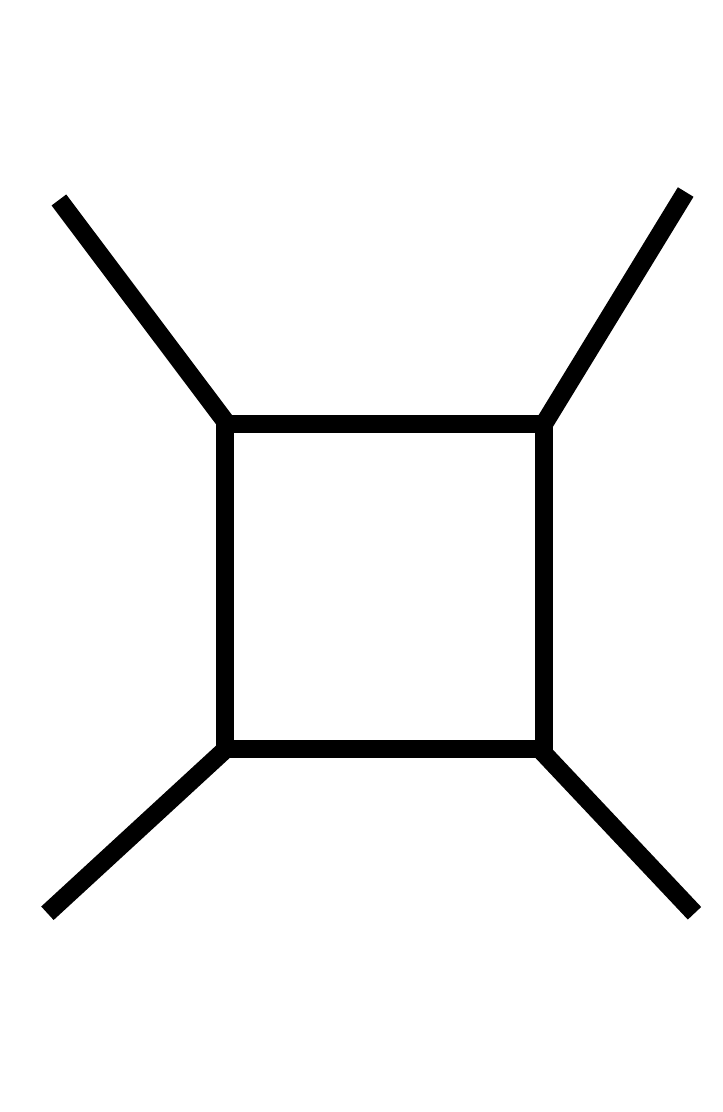}=a\left( \grb{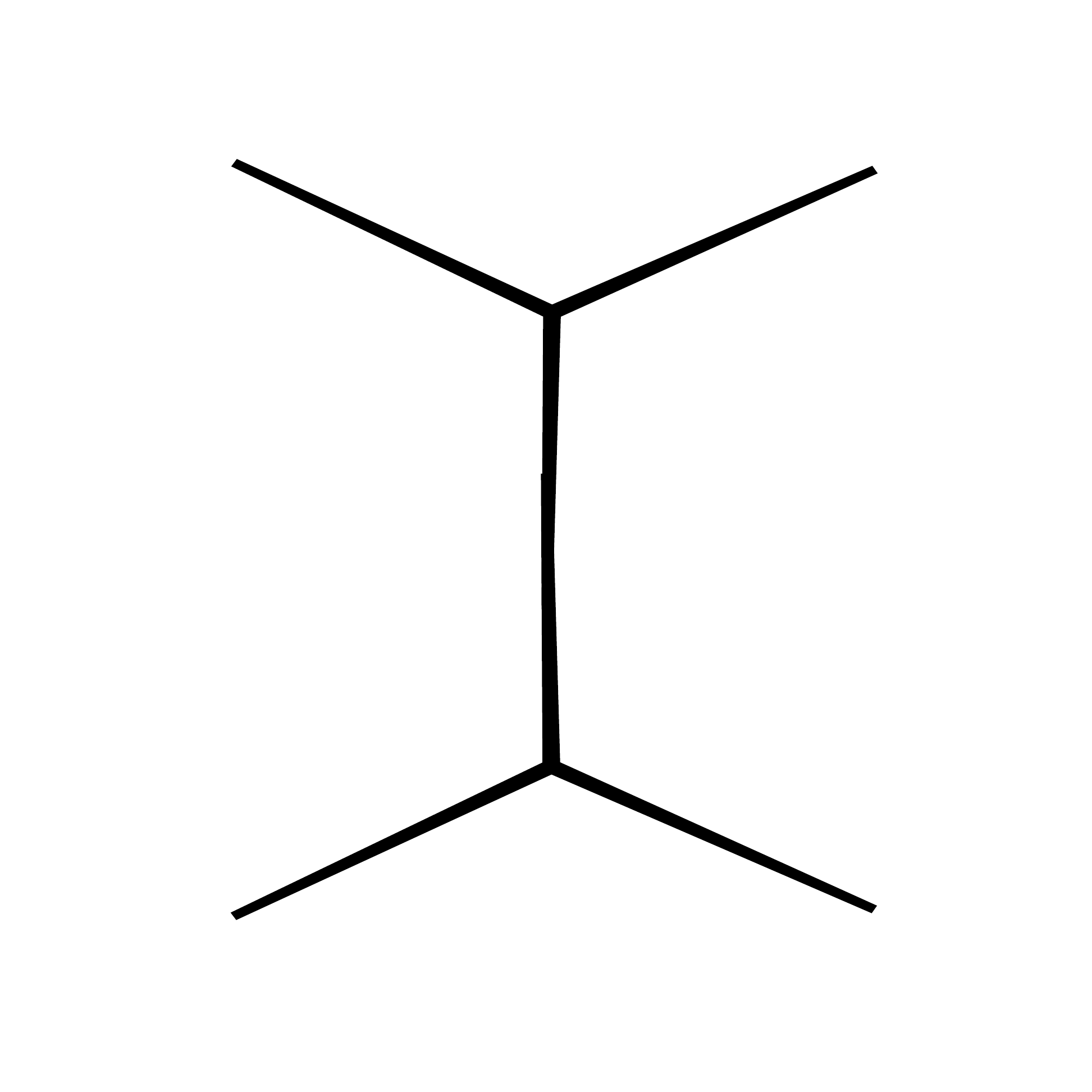}+\grb{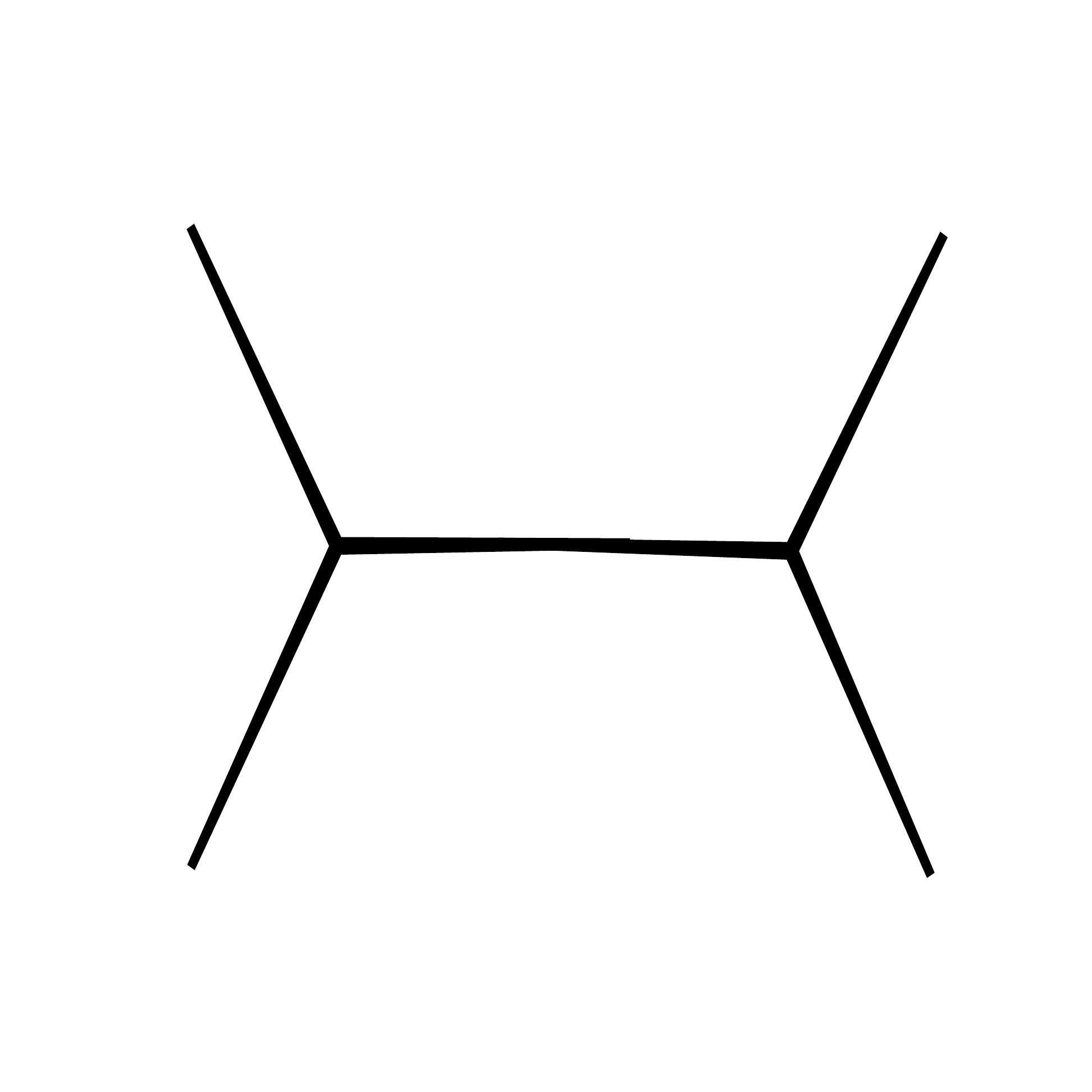}\right)+b \left(\grb{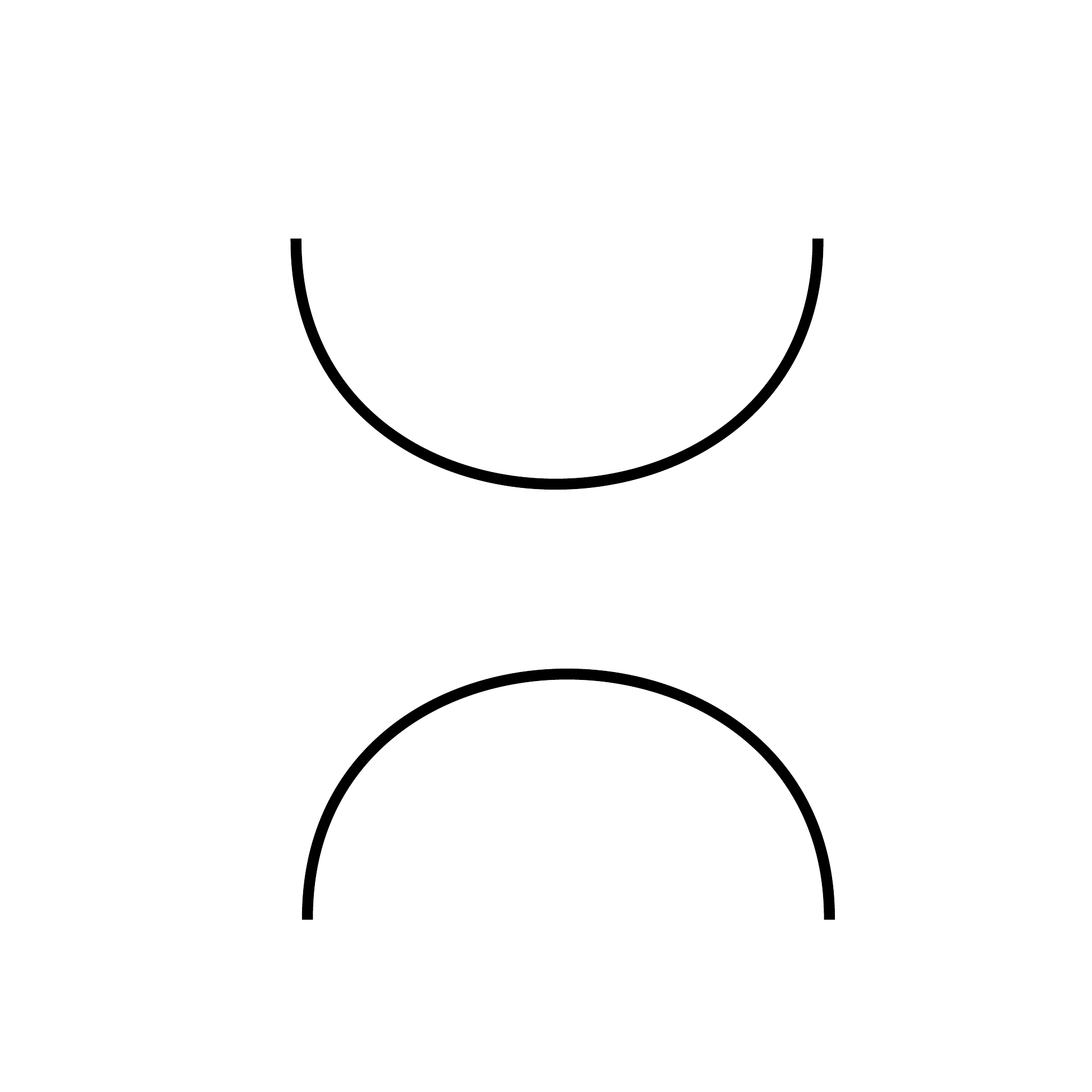}+\grb{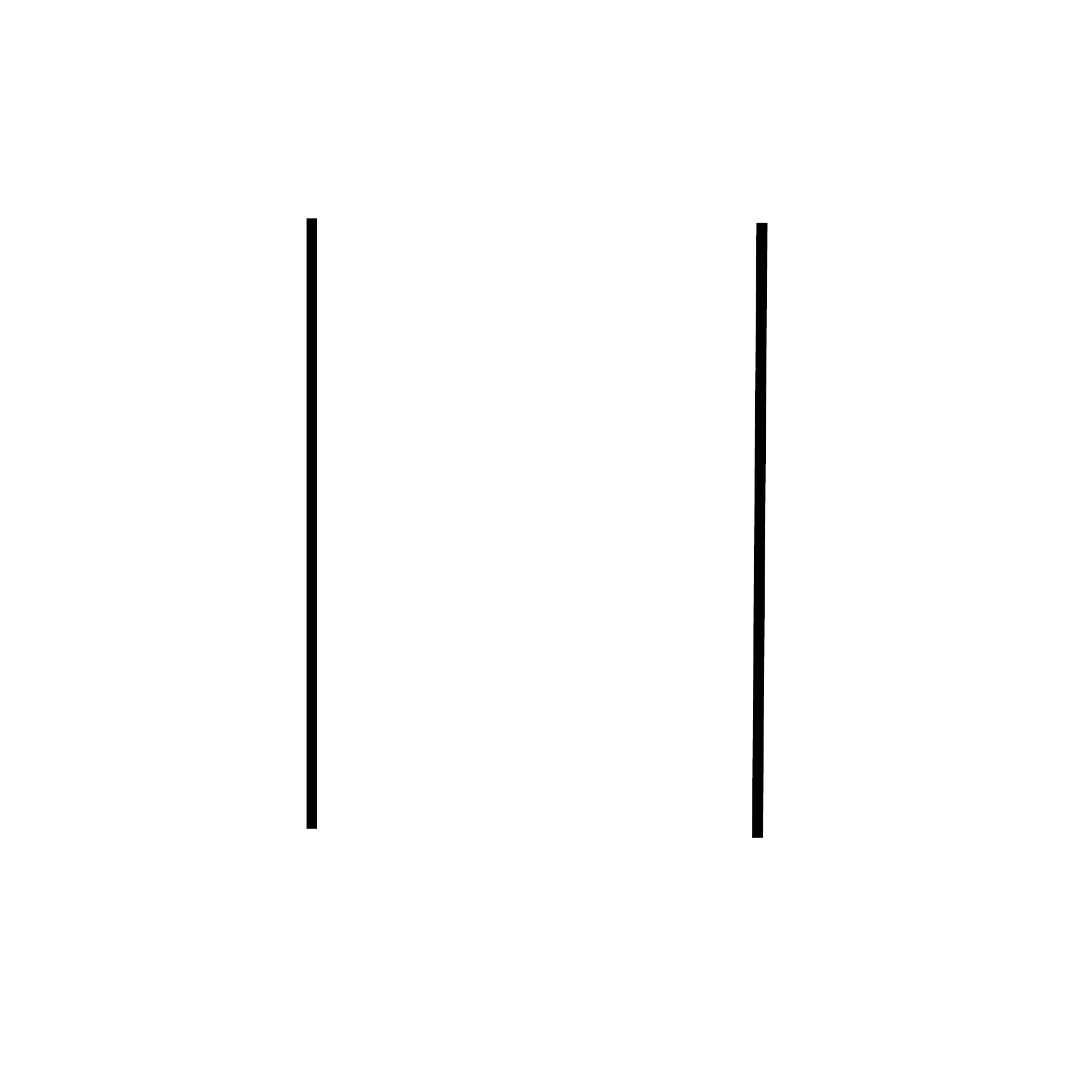}\right)$$
$$\grb{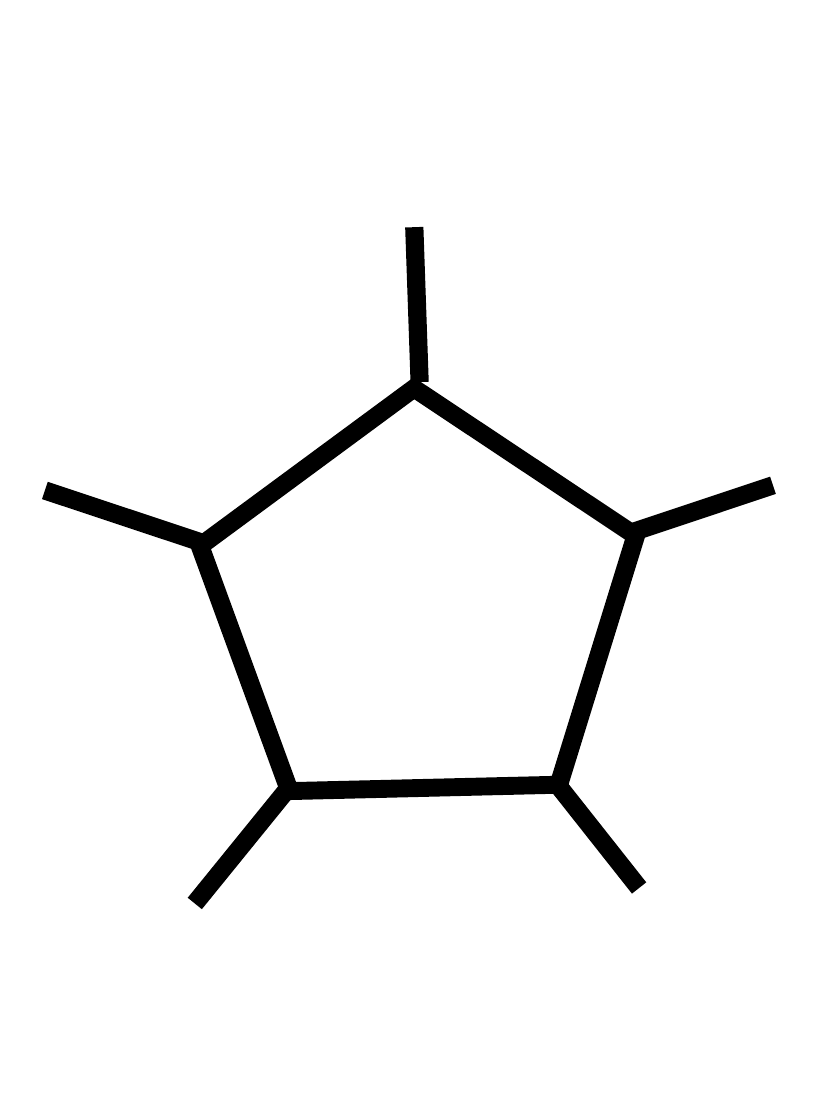}=f\left( \grb{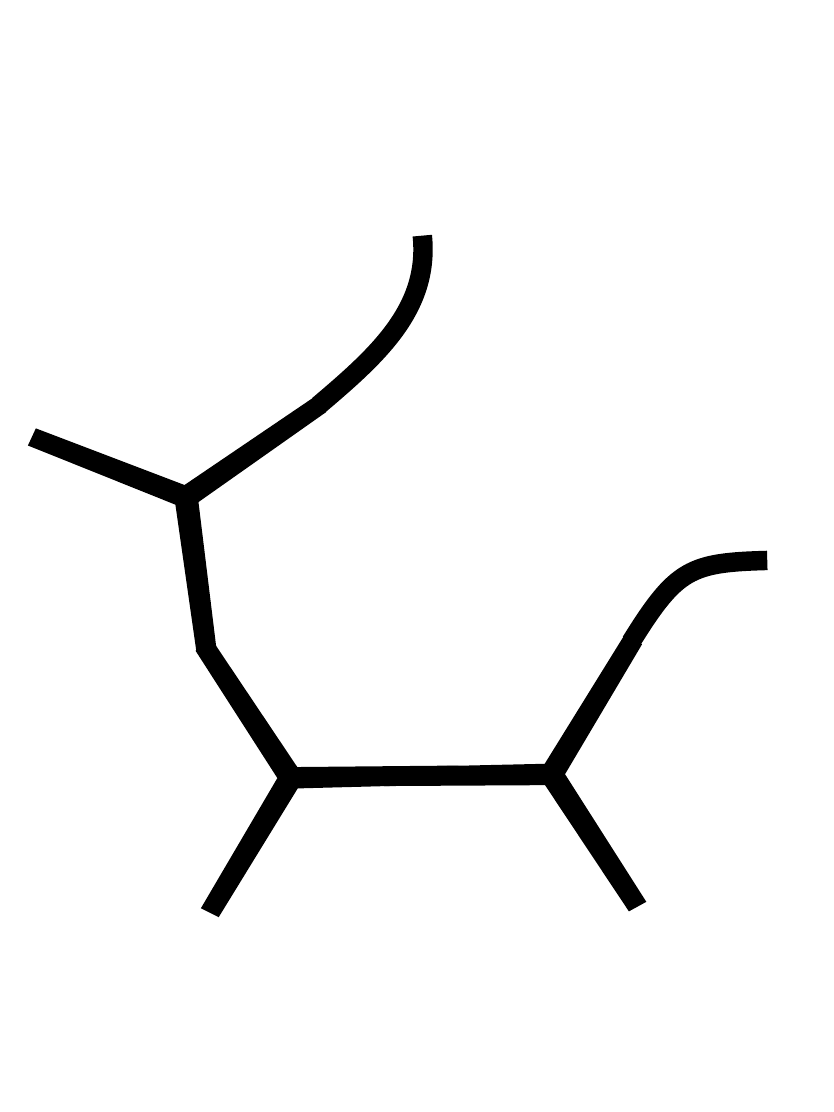}+\text{rotations}\right)+g\left( \grb{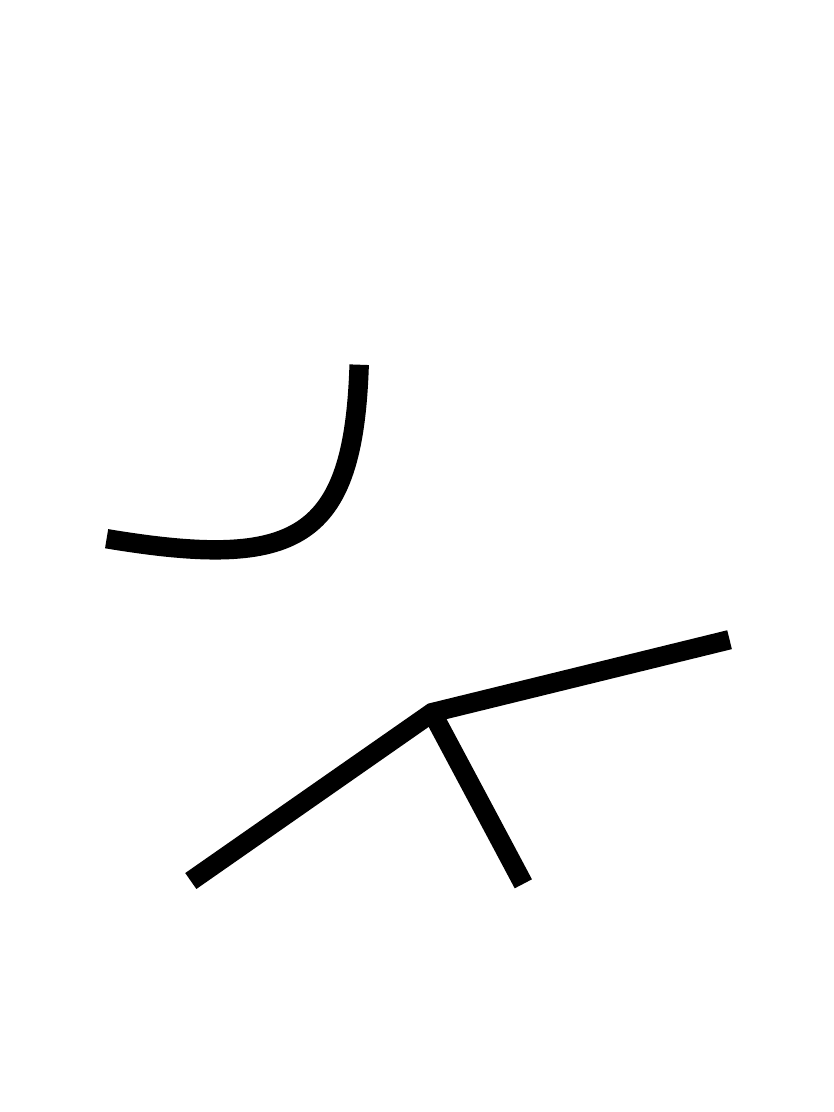}+\text{rotations}  \right)$$

Where

$$a=\frac{q^{2}+q^{-2}}{(q+1+q^{-1})(q-1+q^{-1})(q^{4}+q^{-4})}$$
$$b=\frac{1}{(q+1+q^{-1})(q-1+q^{-1})(q^{4}+q^{-4})^{2}}$$
$$c=-\frac{q^{2}-1+q^{-2}}{q^{4}+q^{-4}}$$
$$f=-\frac{1}{(q+1+q^{-1})(q-1+q^{-1})(q^{4}+q^{-4})}$$
$$g=-\frac{1}{(q+1+q^{-1})^{2}(q-1+q^{-1})^{2}(q^{4}+q^{-4})^{2}}$$

\medskip

\cite{K}, \cite{K2}, and \cite{MPS} shows this category is actually spherical.  The duality maps $\cup$ and $\cap$ provide standard solutions for the simple object (minimal projection) spanning $Mor(1,1)$.  This is the object $X$, which tensor generates our category.

 Kuperberg showed that this category is isomorphic (not just equivalent) to the spherical category generated by the $7$-dimensional fundamental representation (which we also call $X$) of $U_{q}(\mathfrak{g}_{2})$ (see \cite{K2} Theorem 5.1). A single string corresponds to the object $X$ in $Rep(U_{q}(\mathfrak{g}_{2}))$, hence the natural number $k$ an an object in $(G_{2})_{q}$ corresponds to the object $X^{\otimes k}$ in $Rep(U_{q}(\mathfrak{g}_{2}))$.  Since $X$ tensor generates $Rep(U_{q}(\mathfrak{g}_{2}))$, we have the whole category appearing.

 In both Kuperberg's work and Morrison, Peters, and Snyder's no $*$-structure is considered.  However, $U_{q}(\mathfrak{g}_{2})$ has a natural $*$-structure for positive $q\ne 1$ (along with all Drinfeld-Jimbo quantum groups), and it is shown, for example, in \cite{NT}, Chapter $2.4$, that the category of finite dimensional $*$-representations is a rigid $C^{*}$-tensor category.  Every type 1 finite dimensional representation of $U_{q}(\mathfrak{g}_{2})$ for  $q>0$ is unitarizable (\cite{CP}, Chapter 10) and thus the rigid $C^{*}$-tensor category of finite dimensional unitary type 1 representations $\mathscr{C}_{q}(\mathfrak{g}_2)$ is monoidally equivalent to $Rep(U_{q}(\mathfrak{g}_{2}))$.  In the appendix, we show that the $*$-structure from quantum groups transports to $(G_{2})_{q}$ as the trivalent $*$-structure  defined by reflecting a diagram across a horizontal line (see Proposition 5.1 in the Appendix).  Thus we can consider $(G_{2})_{q}$ as a $C^{*}$-trivalent category.

To prove $(G_{2})_{q}$ has property (T), we need to know the structure of $Mor(2,2)$.  $Mor(2,2)$ is a 4-dimensional abelian $C^{*}$-algebra.  To determine the minimal projections, we set $$\xi:=\sqrt{\delta^{2}c^{4}+2\delta(c^{4}-2c^{3}-c^{2}+4c+2)+(c^{2}-2c-1)^{2}}=\frac{(1+q^{2})^{2}(1-q^{2}+q^{6}-q^{8}+q^{10}-q^{14}+q^{16})}{q^{6}(1+q^{8})}.$$

This is manifestly non-zero for $q\ne 1$ and $q>0$.  

 \begin{prop}(\cite{MPS}, Proposition $4.16$) The minimal idempotents in the finite dimensional abelian algebra $Mor(2,2)$  are given by  
$$\frac{1}{\delta}\gra{e},\ \gra{i}$$ 
and the two idempotents

$y_{\pm}=\frac{-(\delta+1)c^{2}\pm \xi+1}{\pm 2\xi} \gra{id}+\frac{\delta(c^{2}-2c-2)\mp \xi +c^{2}-2c-1}{\pm 2\delta \xi} \gra{e} -\frac{\delta(c+2)c\pm \xi +c^{2}+1}{\pm 2\xi} \gra{i}+\frac{\delta c+\delta+c}{\pm \xi} \gra{h}.$

\end{prop}

\medskip

In our setting, we see that the idempotents are in fact projections, since our basis is self-adjoint and all coefficients are real numbers.  These projections correspond to simple objects in the rigid $C^{*}$-tensor category underlying $(G_{2})_{q}$. 

By Kuperberg's isomorphism, the fusion algebra of the underlying projection category of $(G_{2})_{q}$ is isomorphic to the fusion algebra of the category $\mathscr{C}{q}(\mathfrak{g}_{2})$ for positive $q\ne 1$, which in turn is isomorphic to the complexification of the representation ring $R(G_2)$.  It is well known that for compact, simply connected, simple Lie groups $G$, the representation ring $R(G)$ is isomorphic to the ring of polynomials in the fundamental representations.  For a specific reference for $G_2$ see \cite{FH}, and in general, see \cite{Ad}, Theorem 6.41.   This implies the fusion algebra of $(G_{2})_{q}$  is the  (commutative) complex polynomial algebra in $2$ self-adjoint variables $\mathbb{C}[Z_1, Z_2]$, where $Z_1$ and $Z_2$ correspond to the $14$ and $7$ dimensional fundamental representations of the quantum group $U_{q}(\mathfrak{g}_{2})$ respectively.  (Note that self-adjointness of the variables follows from self-duality of the corresponding representations). 

The fusion graph with respect to $X$ is given by

$$\gre{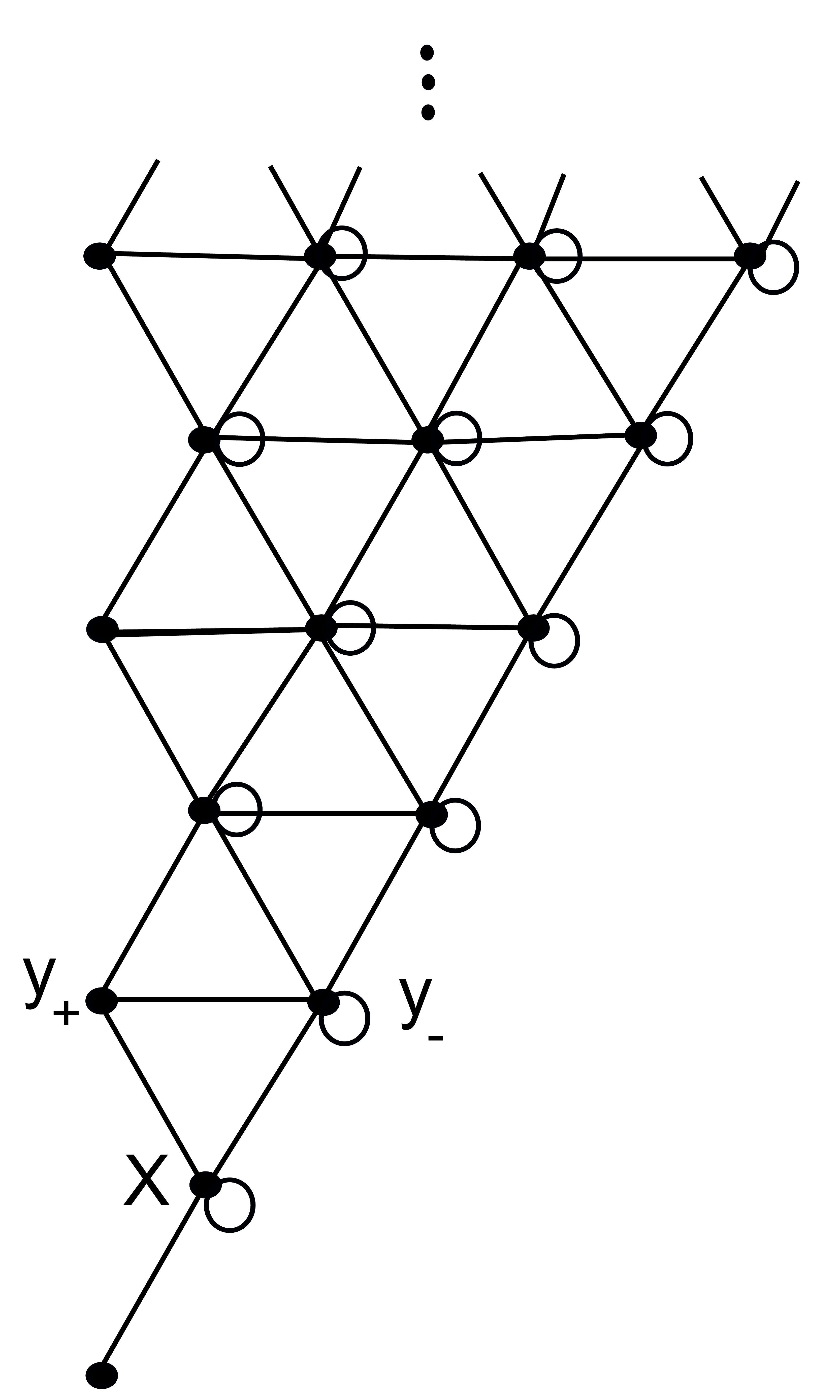}$$

Here, the vertex at the bottom corresponds to the identity, the next highest vertex corresponds to $X$ itself, etc. 

\begin{subsection}{Property (T) for $(G_{2})_{q}$}

Now we consider the tube algebra $\AC$ of the categories $(G_{2})_{q}$ for some positive $q\ne 1$.  In our analysis all such $q$ will yield the same results, so we supress the dependence of the tube algebra $\AC$ on $q$ for notational convenience.  Recall simple objects in the category correspond to minimal projections in some $Mor(k,k)$.  Let us choose our set of representatives of projections so that it contains the empty diagram $id$, the single string $X$, and the two projections $y_{+}$ and $y_{-}$, representing their equivalence classes.  For $x\in Mor(k,k)$, we let $i:Mor(k,k)\rightarrow Mor(k\otimes id, id\otimes k)$ be the canonical identification. Then define $$\Delta(x):=\Psi(i(x))\in \AC_{0,0}.$$
where $\Psi$ is defined in the discussion of the tube algebra.  In our setting we see that the map in Proposition 2.1 is defined by applying $\Delta$ to a projection.  

Translating the fusion algebra description to our setting, the variable $Z_2$ is represented by $\Delta(X)$, while $Z_{1}$ is represented by the projection $\Delta(y_{+})\in Mor(2,2)$. We see that $$\AC_{0,0}\cong \mathbb{C}[\Delta(y_{+}),\Delta(X)].$$

Going back to our expression for $y_{+}$, we see that $\frac{\delta c+\delta+c}{ \xi}=\frac{(1+q^{2}+q^{4})(1+q^{8})}{q^{4}(1+q^{2})^{2}}\ne 0$.  Thus 

\medskip

$\gra{h}=\frac{q^{4}(1+q^{2})^{2}}{(1+q^{2}+q^{4})(1+q^{8})}\left((y_{+})-\frac{-(\delta+1)c^{2}+ \xi+1}{ 2\xi}
 \gra{id}-\frac{\delta(c^{2}-2c-2)- \xi +c^{2}-2c-1}{ 2\delta \xi} \gra{e} +\frac{\delta(c+2)c+ \xi +c^{2}+1}{ 2\xi} \gra{i}\right).$

\medskip
Then since $\Delta(\gra{i})=\Delta(X)$, we have

 $\Delta(\gra{h})=\frac{q^{4}(1+q^{2})^{2}}{(1+q^{2}+q^{4})(1+q^{8})}\left(\Delta(y_{+})-\frac{-(\delta+1)c^{2}+ \xi+1}{ 2\xi}\Delta(X)^{2}
-\frac{\delta(c^{2}-2c-2)- \xi +c^{2}-2c-1}{ 2 \xi} 1+ \frac{\delta(c+2)c+ \xi +c^{2}+1}{ 2\xi} \Delta(X)\right).$

 We denote $H:=\gra{h}$.  Since our polynomial expression for $\Delta(H)$ is linear in $\Delta(y_{+})$ and the terms with powers of $\Delta(X)$ contain no $\Delta(y_{+})$ terms, we can perform an invertible transformation implementing a change of basis, and write an arbitrary polynomial in $\Delta(y_{+})$ and $ \Delta(X)$ as a polynomial in $\Delta(H)$ and $\Delta(X)$, so that
 
 $$\AC_{0,0}\cong \mathbb{C}[\Delta(H), \Delta(X)].$$

Therefore irreducible representations of $\AC_{0,0}$ are  1-dimensional, and they are defined by assigning numbers to $\Delta(H)$ and $\Delta(X)$.  Let us denote by $\alpha$ the value assigned to $\Delta(H)$ and $t$ the value assigned to $\Delta(X)$ in our $1$-dimensional representation.  Let $\gamma_{\alpha, t}: \AC_{0,0}\rightarrow \C$ denote the 1-dimensional representation viewed as a functional, given by evaluating polynomials in $\C[\Delta(H), \Delta(X)]$ at the point $(\alpha,t)$.

 The key point is that while arbitrary values of $\alpha$ and $t$ determine a representation of $\AC_{0,0}$, not all are annular states (hence admissible representations).   Recall that $\gamma_{\alpha,t}$ is admissible if and only if $\gamma_{\alpha,t}(x^{\#}\cdot x)\ge 0$ for all $x\in \AC_{0,k}$ and for all $k\in \text{Irr}((G_{2})_{q})$.

 As a first restriction, for our representation to be admissible, $t\in \mathbb{R}$ since the object corresponding to a single string is self-dual and our representation must be a $*$-representation.  We also  must have $\alpha\ge 0$, since $\Delta(H)=T^{\#}\cdot T$, where $T\in Mor(X\otimes id, X\otimes X)\subseteq \AC^{X}_{0,X}$ is given by the trivalent vertex $T:=\gra{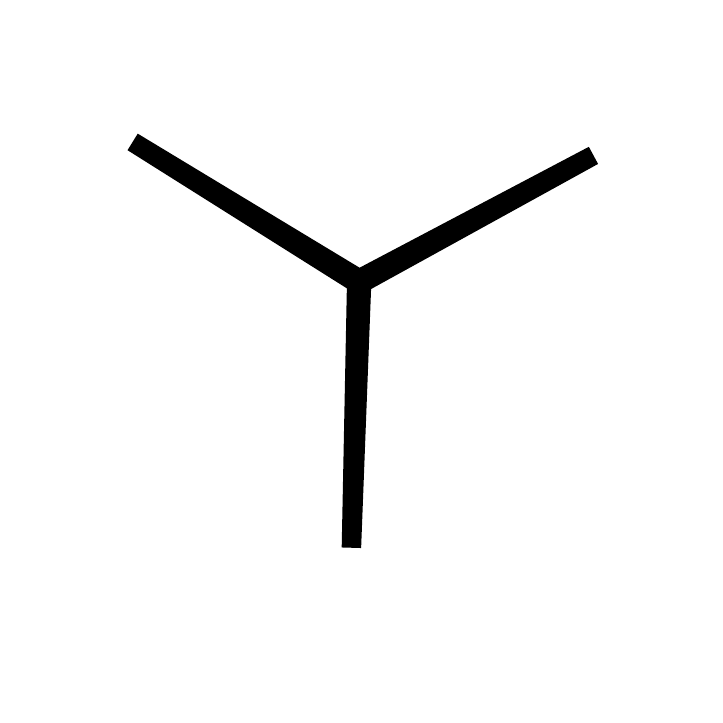}$

\ \ We know also that $|t|\le \delta$ by \cite{PV}, Propsoition 4.2, or \cite{GJ}, Lemma 4.4.  Here $\delta$ is the value of the closed circle defined in terms of $q$ in the description of the skein theory. This restricts the possible one dimensional admissible representations to some subset $Z\subseteq \{(\alpha, t)\subseteq \mathbb{R}^{2}\ :\ \alpha\ge 0,\ t\in [-\delta, \delta]\}$.  

Since the fusion algebra is isomorphic to the polynomial algebra in two self adjoint variables, and irreducible representations correspond to evaluation at points $Z\subseteq \mathbb{R}^{2}$, the weak-$*$ topology on $Z$ as linear functionals on $C^{*}((G_{2})_{q})$ agrees with the (subspace) topology on the plane.  The trivial representation corresponds to the point $(0, \delta)$.  We will show that for positive $q\ne1$, there is a neighborhood of the point $(0,\delta)$ in the rectangle $\mathbb{R}^{+}\times [-\delta, \delta]$ such that the functional $\gamma_{\alpha, t}$ is not an annular state.  

To see this, let $s:=\grb{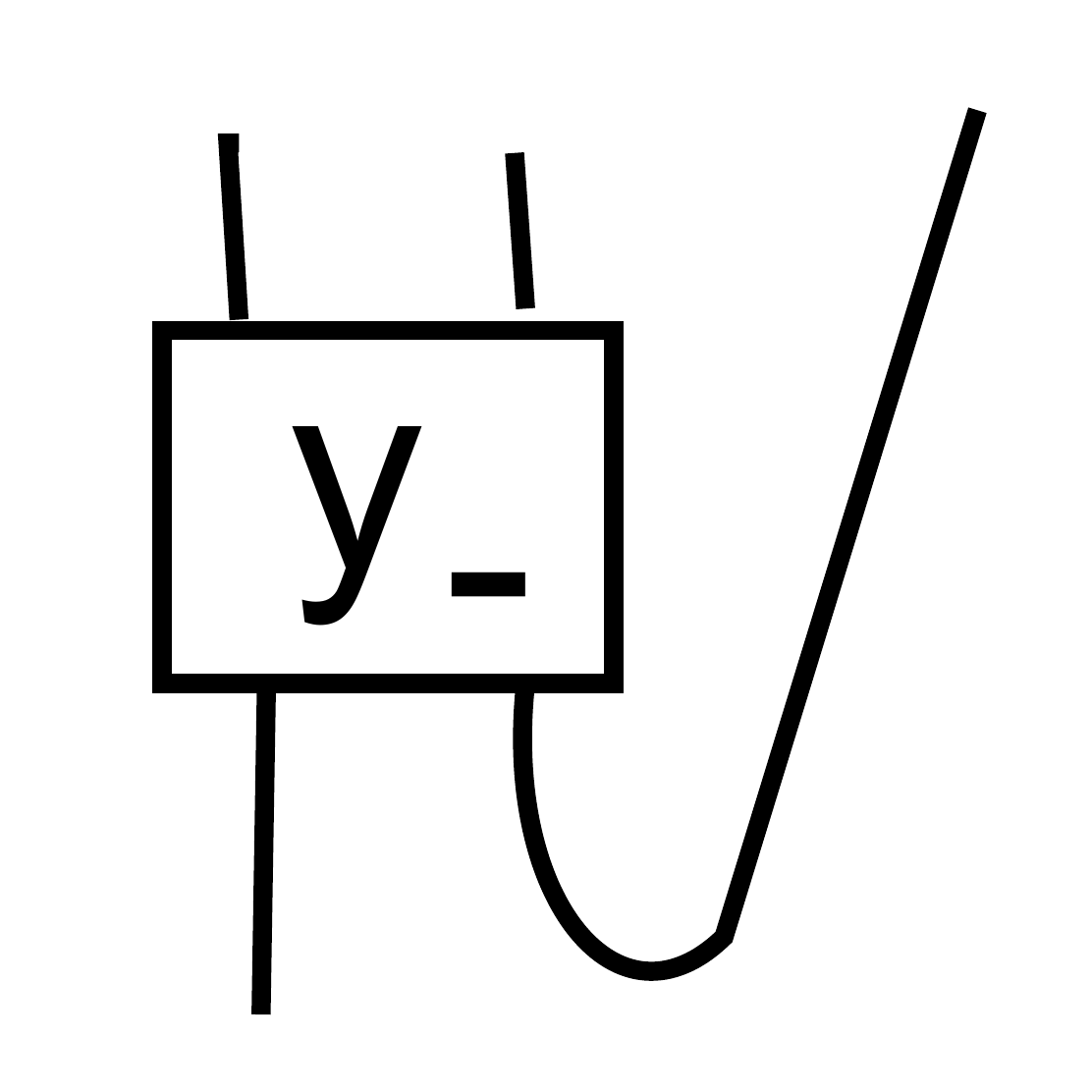}\in Mor(X\otimes id,\ y_{-}\otimes X)$.  We view $s\in\AC^{X}_{0,y_{-}}\subset \AC$.

 For each pair $(\alpha, t)$, define the function $f(\alpha, t):=\gamma_{\alpha, t}(s^{\#}\cdot s)$ . This can be directly computed from the representation of $y_{-}$ in terms of our planar algebra basis, and we obtain

$$f(\alpha, t)=  \delta \frac{-(\delta+1)c^{2}- \xi+1}{- 2\xi}+t^{2}\frac{\delta(c^{2}-2c-2)+ \xi +c^{2}-2c-1}{- 2\delta \xi}-\alpha \frac{\delta(c+2)c- \xi +c^{2}+1}{- 2\xi} +t \frac{\delta c+\delta+c}{- \xi}.$$

By construction, if the functional corresponding to $(\alpha, t)$ is an annular  state, $f(\alpha, t)$ must be non-negative. 

\begin{prop} For all positive $q\ne 1$, $(G_{2})_{q}$ has property $(T)$.
\end{prop}

\begin{proof}
Since $y_{-}$ is a minimal projection in $Mor(2,2)$ and is not equivalent to $id$ in $(G_2)_{q}$, $f(0, \delta)=0$.  This can also be seen by direct computation.  Let $v:=(x,y)\in \mathbb{R}^{2}$ be a non-zero vector in the fourth quadrant of the plane (including the axes), or in other words $x\ge 0$ and $y\le 0$, but $(x,y)\ne (0,0)$.  We wish to show that $f(x, \delta+y)<0$ for sufficiently small $||v||$.  This will demonstrate that in a neighborhood of $(0,\delta)$ in $\mathbb{R}^{+}\times [-\delta, \delta]$, the function $f(\alpha, t)$ will be strictly negative, hence the representation corresponding to $\gamma_{\alpha, t}$ is not admissible.


We see that for positive $q\ne 1$,

 $$\frac{\partial f}{\partial \alpha}|_{(0,\delta)}=\frac{\delta(c+2)c-\xi +c^{2}+1}{2\xi}=-\frac{1+q^{2}+2q^{4}+q^{6}+q^{8}}{(q+q^{3})^{2}}<0.$$

 This is always strictly negative (for $q\ne 0$). We compute
$$\frac{\partial f}{\partial t}|_{(0,\delta)}=\frac{(\delta+1)(c^{2}-c-1)}{-\xi}-1=\frac{(-1+q^{2})^{2}(1+q^{2}+q^{4})}{q^{4}}>0$$

This expression is strictly positive for all $q\ne 1, q\ne 0$.  Therefore we have that the directional derivative $\frac{\partial f}{\partial v}|_{(0,\delta)}<0$ for $v$ in the prescribed range. 
We remark that for $q=1$, $\frac{\partial f}{\partial t}|_{(0,\delta)}=0$, hence this part of our proof breaks down as expected, since $Rep(G_2)$ is amenable.

Letting $B$ denote the compact set of unit vectors in the fourth quadrant, since $\frac{\partial f}{\partial v}|_{(0,\delta)}$ is a continuous function of $v$, there exists some $M<0$ such that $\frac{\partial f}{\partial v}|_{(0,\delta)}\le M<0$ for $v\in B$.

  Now, it is straightforward to compute

$$\frac{\partial^{2} f}{\partial^{2} t}|_{(0,\delta)}=\frac{2q^{6}(1+q^{2}+q^{4})}{(1+q^{2})^{2}(1+q^{2}+q^{4}+q^{6}+q^{8}+q^{10}+q^{12})}>0,$$

\medskip

 and it is easy to see that all other second order partial derivatives are $0$, and all higher order derivatives with respect to both variables are $0$. Then by Taylor's theorem, we have
$$f(x, \delta+y)=x\frac{\partial f}{\partial \alpha}|_{(0,\delta)}+y\frac{\partial f}{\partial t}|_{(0,\delta)}+\frac{y^{2}}{2}\frac{\partial^{2} f}{\partial^{2} t}|_{(0,\delta)}$$
for arbitrary $v=(x,y)$ in the fourth quadrant.  Let $v^{\prime}=\frac{1}{\|v\| }v$.  Since $y^{2}\le \|v\|^{2}$,  setting $\lambda:=\frac{1}{2} \frac{\partial^{2} f}{\partial^{2} t}|_{(0,\delta)}$ gives

$$f(x, \delta+y)=\|v\|\frac{\partial f}{\partial v^{\prime}}|_{(0,\delta)}+y^{2}\lambda\le \|v\| M+\|v\|^{2}\lambda=\|v\|(M +\|v\| \lambda).$$

If we set $\epsilon=\frac{|M|}{\lambda}$, then since $M<0$, for $0<\|v\|<\epsilon$, we see that $f(x, y+\delta)<0$.  Therefore $(G_{2})_{q}$ has property (T) for positive $q\ne1$.
\end{proof}

\end{subsection}

\end{section}

\begin{section}{Concluding Remarks}
\begin{enumerate}
\item
Our result provides another class of examples of subfactors with property $(T)$ standard invariant.  To see this, we recall that Popa gave an axiomatization of standard invariants of subfactors as standard $\lambda$-lattices \cite{Po2}. These are towers unital inclusions of finite dimensional $C^{*}$- algebras, together with some extra data.  From a rigid $C^{*}$-tensor $\ca$ and object $X$, one can construct a standard $\lambda$-lattice by

$$\begin{array}{cccccccc} 
\C & \subset & End(X) & \subset & End(X\overline{X}) & \subset & End(X\overline{X}X) & \subset \cdot \cdot \cdot \\
  &   &  \cup &  & \cup &  & \cup & \\
 &  & \C & \subset & End(\overline{X}) & \subset & End(\overline{X}X) & \subset \cdot \cdot \cdot \\ \\
\end{array}$$

Then this tower will be the standard invariant of some subfactor $N\subseteq M$ by \cite{Po2}.  The category of $N$-$N$ bimodules will be isomorphic to the subcategory of $\ca$ generated by $X\overline{X}$.  Applying this construction to $(G_{2})_{q}$, since $X$ is self-dual and appears as a sub-object of $X\otimes X$, the even bimodules of this subfactor will be a category equivalent to $\mathscr{C}_{q}(\mathfrak{g}_{2})$ (as opposed to a proper sub-category).  Therefore, as in \cite{PV}, Theorem 8.1, this subfactor will have property (T) standard invariant.

\medskip

\item
We hope that methods similar to those presented here will be useful to deduce property (T) (or lack thereof) for other quantum group categories which have a nice planar algebra description, particularly the $BMW$ planar algebras, which describe the quantum $SO(n)$ and $SP(2n)$ categories.  
\end{enumerate}
\end{section}

\begin{section}{APPENDIX: $*$-structure for $(G_{2})_{q}$}

The point of this section is to show that the $*$-structure we described above for $(G_{2})_{q}$ (reflection of diagrams about a horizontal line) gives a $C^{*}$-trivalent category.  To do this from the purely planar algebra perspective is quite a daunting task, since we would have to explicitly construct all idempotents, show they are self adjoint, and that they have positive trace.  Fortunately for us, due to Kuperberg's isomorphism, $(G_{2})_{q}$ can be realized as the category generated by the fundamental $7$-dimensional representation $X$ of the Hopf $*$-algebra $U_{q}(\mathfrak{g}_{2})$, which carries with it a naturally occurring $*$-structure.  

 To elaborate, we know that quantum groups $U_{q}(\mathfrak{g})$ with positive $q$ have a natural $*$-structure and that every (type 1) finite dimensional representation is unitarizable, which means it is equivalent to a $*$-representation of the Hopf $*$-algebra (see \cite{CP}, Chapter 10).  To obtain a rigid $C^{*}$-tensor category, we restrict our attention to the finite dimensional type 1 unitary representations, and since every representation is unitarizable, this category is monoidally equivalent to the whole category of finite dimensional type 1 representations. 

 Kuperberg's equivalence from $Rep(U_{q}(\mathfrak{g}_{2}))$ to $(G_{2})_{q}$ used the fact that the fundamental $7$-dimensional representation was symmetrically self-dual.  But being symmetrically self-dual depends on the specific choice of duality maps and most importantly on the map implementing the equivalence from $\pi$ to $\pi^{c}$ (the standard dual, defined using the Hopf algebra antipode).  For us to have a $C^{*}$-planar algebra, we need $\pi$ to be \textit{unitarily} symmetrically self dual.  Even  if $\pi$ is unitary, the canonical dual $\pi^{c}$ is not necessarily unitary, although we know it is equivalent to a unitary representation.  In \cite{NT}, they explicitly identify the unitary dual, $\overline{\pi}$ for all representations of $U_{q}(\mathfrak{g})$.  Using this information, we explicitly compute the unitary intertwiner $T\in (\pi, \overline{\pi})$ in the matrix representations of $\pi$ and $\overline{\pi}$, and show that $T$ along with the choices of (mutually adjoint) duality maps implements a unitary symmetric self-duality on $\pi$.  

Furthermore, these duality maps provide our rigid $C^{*}$ tensor category with a pivotal structure, such that duality is compatible with the $*$-structure.  Using Kuperberg's result, our category (forgetting the $*$-structure) must be the trivalent category described by our skein theory. Then, compatibility of the $*$-structure with the duality functor forces the $*$-structure described in Definition 3.2 for $(G_{2})_{q}$, given by horizontal reflection of diagrams.

\medskip

 \ \ We give a brief description of the Drinfeld-Jimbo Hopf $*$-algebra $U_{q}(\mathfrak{g}_{2})$, following \cite{NT}, Definition 2.4.1.  Let $\alpha_{1}, \alpha_{2}$ be the standard choice of simple roots of the $G_{2}$ root system, pictured below:

$$\gre{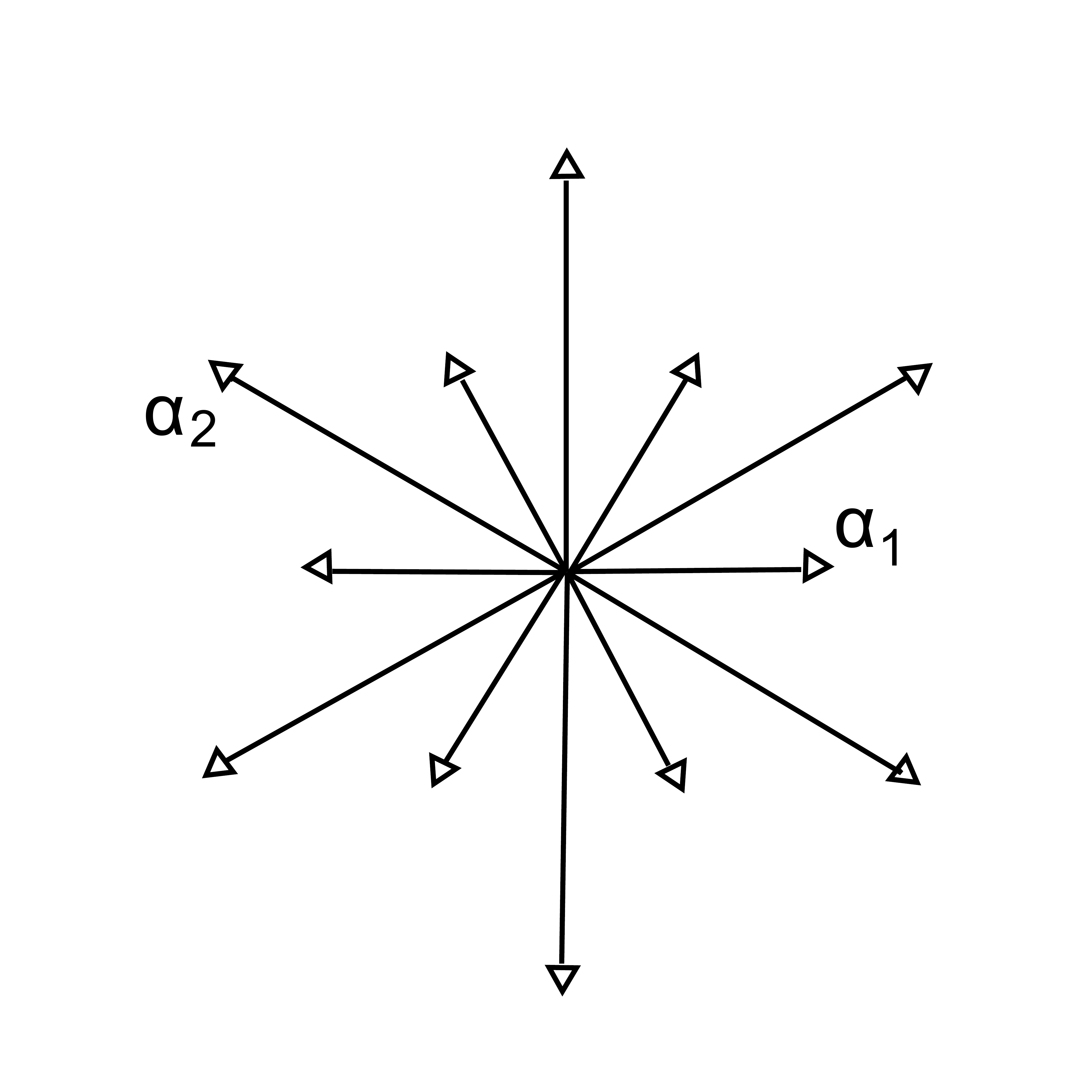}$$ 

Let $(a_{ij})=\left( \begin{array}{cc} 2 & -3 \\ -1 & 2\end{array} \right)$ be the Cartan matrix for the $G_{2}$ root system.  Let $d_{1}:=1, d_{2}:=3$, so that $A_{ij}:=d_{i}a_{i,j}=(\alpha_{i},\alpha_{j})$ is the inner product matrix, given by $$A=\left( \begin{array}{cc} 2 & -3 \\ -3 & 6 \end{array}\right)$$

\ \ For $q> 0$, $q\ne 1$, let $q_{i}=q^{d_{i}}$, $i=1,2$.  Then $U_{q}(\mathfrak{g}_{2})$ is defined as the universal unital algebra generated by elements $E_{i}, F_{i} K_{i}, K^{-1}_{i},\ i=1,2$, satisfying the relations

$$K_{i}K^{-1}_{i}=K^{-1}_{i}K_{i}=1,\ K_{i}K_{j}=K_{j}K_{i}$$
$$K_{i}E_{j}K^{-1}_{i}=q^{a_{ij}}_{i} E_{j},\ \ K_{i}F_{j}K^{-1}_{i}=q^{-a_{ij}}_{i} F_{j}$$
$$[E_{i},F_{j}]=\delta_{ij}\frac{K_{i}-K^{-1}_{i}}{q_{i}-q^{-1}_{i}}$$
$$\text{For}\ i\ne j\ \sum^{1-a_{ij}}_{k=0}(-1)^{k}\left[ \begin{array}{c} 1-a_{ij} \\ k\end{array} \right]_{q_{i}} E^{k}_{i}E_{j}E^{1-a_{ij}-k}_{i}=0$$
$$\text{and}\ \sum^{1-a_{ij}}_{k=0}(-1)^{k}\left[ \begin{array}{c} 1-a_{ij} \\ k\end{array} \right]_{q_{i}} F^{k}_{i}F_{j}F^{1-a_{ij}-k}_{i}=0$$

where $\left[ \begin{array}{c} m \\ k \end{array} \right]_{q_{i}}=\frac{[m]_{q_i} !}{[k]_{q_{i}} ! [m-k]_{q_i} !}$, $[m]_{q_i} !=[m]_{q_i} [m-1]_{q_i}\ \dots [1]_{q_{i}}$, and $[n]_{q_i}=\frac{q^{n}_{i}-q^{-n}_{i}}{q_{i}-q^{-1}_{i}}$.

This algebra is a Hopf $*$-algebra with coproduct $\hat{\triangle}_{q}$ define by 

$$\hat{\triangle}_{q}(K_{i})=K_{i}\otimes K_{i},\ \hat{\triangle}_{q}(E_{i})=E_{i}\otimes 1+K_{i}\otimes E_{i},\ \hat{\triangle}_{q}(F_{i})=F_{i}\otimes K^{-1}_{i}+1\otimes F_{i}$$

and with involution given by $K^{*}_{i}=K_{i},\ E^{*}_{i}=F_{i}K_{i},\ F^{*}_{i}=K^{-1}_{i} E_{i}$.

\medskip

 \ \ The counit $\hat{\epsilon}_{q}$ and the antipode $\hat{S}_{q}$ are given by 
$$\hat{\epsilon}_{q}(K_{i})=1,\ \hat{\epsilon}_{q}(E_{i})=\hat{\epsilon}_{q}(F_{i})=0,$$
$$\hat{S}_{q}(K_{i})=K^{-1}_{i},\ \hat{S}_{q}(E_{i})=-K^{-1}_{i}E_{i},\ \hat{S}_{q}(F_{i})=-F_{i}K_{i}$$

A representation is type 1 if it decomposes as the direct sum of weight spaces (see \cite{NT}, Definition 2.4.3).  The category of type 1 finite dimensional representations of this algebra has the structure of a rigid tensor category, with tensor product defined using the comultiplication, and duality the vector space dual with action induced by the antipode.  This is a standard result for Hopf algebras.  A \textit{unitary representation} (or $*$-representation) is simply a $*$-algebra homomorphism $\pi: U_{q}(\mathfrak{g})\rightarrow B(H)$ for some Hilbert space $H$.  We consider here only finite dimensional type 1 unitary representations.  This category is a rigid $C^{*}$-tensor category. We refer the reader to \cite{NT}, Chapter 2.4 for details.  As we have mentioned above, every finite dimensional representation of these Hopf algebras is unitarizable (see \cite{CP}, Chapter 10 ) for positive $q$, so that the category of unitary representations is monoidally equivalent to the category of all finite dimensional representations.   A key point in quantum group theory is that representations for positive $q$ are in one-to-one correspondence with classical representations of the Lie algebras, and have the same dimension (as vector spaces) as the classical representations.

We will give an explicit unitary realization of the fundamental $7$-dimensional representation $X$ of $U_{q}(\mathfrak{g}_{2})$.  Let $v_{0}$ be the highest weight vector, normalized so that $\langle v_{0}, v_{0}\rangle =1$.  Then define $H$ as the Hilbert space with orthonormal basis

$$v_{0},\ v_{1}:=q^{\frac{1}{2}}F_{1}v_{0},\ v_2:=q^{2}F_{2}F_{1}v_{0},\ v_{3}:=q^{3}[2]^{\frac{1}{2}}_{q}F_{1}F_{2}F_{1}v_{0},\ v_{4}:=q^{3}[2]^{-1}_{q}F_{1}F_{1}F_{2}F_{1}v_{0},$$
$$ v_{5}:= q^{\frac{9}{2}}[2]^{-1}F_{2}F_{1}F_{1}F_{2}F_{1}v_{0},\ v_{6}:=q^{5}[2]^{-1}_{q}F_{1}F_{2}F_{1}F_{1}F_{2}F_{1}v_{0}.$$

 The action $\pi$ of $U_{q}(\mathfrak{g}_{2})$ on vectors can be worked out from the commutation relations.  This yields a $*$-representation on which the actions of $E_{1}, E_{2}, F_{1}, F_{2}$ can be worked out explicitly.  For example, we have the matrix representations with respect to the above basis given by

$$\pi(E_{1})=\left( \begin{array}{ccccccc} 
0 & q^{\frac{1}{2}} & 0 & 0 & 0 & 0 & 0 \\
0 & 0 & 0 & 0 & 0 & 0 & 0 \\
0 & 0 & 0 & q[2]^{\frac{1}{2}}_{q} & 0 & 0 & 0 \\
0 & 0 & 0 & 0 & [2]^{\frac{1}{2}}_{q} & 0 & 0 \\
0 & 0 & 0 & 0 & 0 & 0 & 0 \\
0 & 0 & 0 & 0 & 0 & 0 & q^{\frac{1}{2}} \\
0 & 0 & 0 & 0 & 0 & 0 & 0 
\end{array} \right)$$ 

$$\pi(E_{2})=\left( \begin{array}{ccccccc} 
0 & 0 & 0 & 0 & 0 & 0 & 0 \\
0 & 0 & q^{\frac{3}{2}} & 0 & 0 & 0 & 0 \\
0 & 0 & 0 & 0 & 0 & 0 & 0 \\
0 & 0 & 0 & 0 & 0 & 0 & 0 \\
0 & 0 & 0 & 0 & 0 & q^{\frac{3}{2}} & 0 \\
0 & 0 & 0 & 0 & 0 & 0 & 0 \\
0 & 0 & 0 & 0 & 0 & 0 & 0 
\end{array} \right) $$

$$\pi(K_{1})=\text{diag}(q,\ q^{-1},\ q^{2},\ 1,\ q^{-2},\ q,\ q^{-1})\ \text{and}\ \pi(K_{2})=\text{diag}(1,\ q^{3},\ q^{-3},\ 1,\ q^{3},\ q^{-3},\ 1)$$

 \ \ If $f\in B(H,K)$, then define $j(f)\in B(\overline{K}, \overline{H})$ by $j(f)\overline{\xi}=\overline{f^{*}\xi}$ for $\xi\in K$.    Let $\rho$ be the half-sum of positive roots and thus $2\rho$ the sum of positive roots.  Then $2\rho=10\alpha_{1}+6\alpha_{2}$, so we define $K_{2\rho}:=K^{10}_{1} K^{6}_{2}$.  Then we have $\hat{S}^{2}_{q}(x)=K^{-1}_{2\rho} x K_{2\rho}$.  Then we see that $\pi(K_{2\rho})$ is a positive invertible operator (diagonal with respect to our chosen basis), and we define $W:=\pi(K^{-1}_{2\rho})^{\frac{1}{2}}\in B(H)$.  Then $j(W)\in B(\overline{H})$, and we define $\overline{\pi}(\ .\ )=j(W)\pi^{c}(\ .\  )j(W^{-1})$.  This is manifestly equivalent to $\pi^{c}$ but has the advantage of being unitary.  

Alternatively, the dual representation can be given using the unitary antipode $\hat{R}_{q}$, defined by $$\hat{R}_{q}(K_{i})=K^{-1}_{i},\ \hat{R}_{q}(E_{i})=-q_{i}K^{-1}_{i}E_{i},\ \hat{R}_{q}(F_{i})=-q^{-1}_{i}F_{i}K_{i}$$

Then if $\pi:U_{q}(\mathfrak{q}_{2})\rightarrow B(H)$, the unitary dual $\overline{\pi}: U_{q}(\mathfrak{g}_{2})\rightarrow B(\overline{H})$ can be realized as $\overline{\pi}(x)\overline{\xi}=\overline{\pi(\hat{R}_{q}(x))^{*}\xi}$ for $x\in U_{q}(\mathfrak{g} _{2})$.  

(We note the standard dual $\pi^{c}$ is given by the same formula, with the standard antipode $\hat{S}_{q}$ in place of $\hat{R}_{q}$).

 Since $W$ is diagonal in the basis described above, the $i^{th}$ eigenvalue, $W_{i}>0$, is clear.  By inspection, we see that $W_{i}W_{7-i+1}=1$.  Consider the map $T:H\rightarrow \overline{H}$ given by $T(v_{i})=(-1)^{i+1}\overline{v}_{7-i+1}$.  Using the description provided above, it is straightforward to check that $T\in Hom(\pi, \overline{\pi})$.  Furthermore, $T^{*}=T^{-1}$, and is given by the same formula as $T$, with the bars reversed, and we have $T^{*}j(W)=W^{-1}T^{*}$

 Consider the map $\overline{r}\in (\mathbb{C},H\otimes \overline{H})$ given by $\overline{r}(1)=\sum_{i} v_{i}\otimes \overline{v}_{i}$ (this map does not depend on basis).  It is easy to see that $r\in (\epsilon, \pi\otimes \pi^{c})$.  By definition $j(W)\in (\pi^{c},\overline{\pi})$.  Then define $R:=(1\otimes (T^{*} j(W)))\overline{r}=(W\otimes T^{*})r\in (\epsilon, \pi \otimes \pi)$.  We claim that the pair $(R, R^{*})$ provide a standard, symmetrically self-dual solution to the conjugate equations for the object $\pi$.  It is straightforward to check that 

$$ (1_{H}\otimes R^{*}) (R\otimes 1_{H}) =( R^{*}\otimes 1_{H}) (1_{H}\otimes R)=1_H.$$

Now, to see symmetric self duality, we note that since $W$ is self adjoint, $r^{*}\circ(W^{-1}\otimes j(W))=r^{*}$, and $(W^{-1}\otimes j(W))\circ r=r$.  Similarly, if $\overline{r}(1)=\sum_{j}\overline{v}_{j}\otimes v_{j}\in B(\mathbb{C}, \overline{H}\otimes H)$, then since $T$ is a self adjoint unitary, $\overline{r}^{*}\circ (T^{*}\otimes T)=\overline{r}^{*}$ and $(T\otimes T^{*})\circ \overline{r}=\overline{r}$.  Using these relations and the fact that $r, \overline{r}$ themselves solve the duality equations in the category of Hilbert spaces, we have

$$  (1_{H}\otimes 1_{H}\otimes R^{*})\circ (1_{H}\otimes R\otimes 1_{H})\circ  R=R=(R^{*}\otimes 1_{H}\otimes 1_{H})\circ (1_{H}\otimes R\otimes 1_{H}).$$

  We let $X=(\pi, H)$ described above.  We define $\mathcal{C}$ to be the strict (up to Hilbert space associativity) $C^{*}$-tensor category generated by the symmetrically self dual object $X$, with unit object $id:=\hat{\epsilon}_{q}$ (the trivial representation of the quantum group given by the counit), and duality maps compositions of $R$ and $R^{*}$ in the obvious fashion.  Then objects are given by tensor powers $X^{\otimes n}$, where $n\in \mathbb{Z}_{+}$, and morphisms are intertwiners of the corresponding quantum group representations.   It is easy to check that the duality maps we have defined induce a pivotal structure on this category (essentially following from the fact that $(r, \overline{r})$ induce a pivotal structure on finite dimensional Hilbert spaces), hence there is an unambiguously defined dual $\overline{f}\in Mor(X^{n}, X^{m})$ for every $f\in Mor(X^{m}, X^{n})$.  Also we see that our $*$-structure is compatible with duality in the sense that $\overline{f^{*}}=\overline{f}^{*}$, which follows since our solutions to the conjugate equations are mutually adjoint, and from the corresponding property for Hilbert spaces with $(r, \overline{r})$.

\begin{prop} $(G_{2})_{q}$ is a $C^{*}$-trivalent category.
\end{prop}
\begin{proof}
Neglecting the $*$- structure, $\mathcal{C}$ is precisely the $G_{2}$ algebraic spider defined by Kuperberg, hence must be isomorphic to $(G_{2})_{q}$ ( \cite{K}, 3.2).  We now have a $*$-structure on the morphism spaces, which is positive with respect to the trace (since $\mathcal{C}$ is manifestly a $C^{*}$-category).  Note that by construction $\cup^{*}=R^{*}=\cap$.

  Using the fusion rules,  $Mor (X, X\otimes X)$ is one dimensional, and is spanned in the planar algebra by a rotationally invariant vertex $t\in (X, X\otimes X)$.  We have $t$ normalized so that $\overline{t}\circ t=1_{X}$, where $\overline{t}= (1_{X}\otimes R^{*})\circ (1_{X}\otimes 1_{X}\otimes R^{*}\otimes 1_{X})\circ (1_{X}\otimes t\otimes 1_{X}\otimes 1_{X})\circ (R\otimes 1_{X}\otimes 1_{X})\in (X\otimes X, X)$ is the dual of $t$. We know that $t^{*}=\lambda \overline{t}$ for some $\lambda\in \mathbb{C}$.  This is a $C^{*}$-category so $t^{*}t=\lambda 1_{X}$ must be positive, hence $\lambda>0$.  We also have $\overline{t}^{*}=\lambda^{\prime}t $.  But $t=t^{**}=\lambda \lambda^{\prime}t$, hence $\lambda^{\prime}=\lambda^{-1}$.  Since $*$ is compatible duality as described in the paragraph before this proposition, we have that $\overline{t}^{*}=\overline{t^{*}}$, hence $\lambda^{-1}t=\lambda t$, but since $\lambda>0$, it must be that $\lambda=1$.

The $*$ of $t$ and $\cap$ determine the $*$ on an arbitrary diagram in $Mor(k,m)$, and it is simply the reflection about a horizontal line.  Therefore $(G_{2})_{q}$ is a $C^{*}$-trivalent category.
\end{proof}

\end{section}

\end{document}